\documentclass[10pt]{article}
\usepackage{geometry}
\usepackage{graphicx,xcolor,soul}
\usepackage{amsmath,amssymb,amsthm,bbm}
\usepackage[utf8]{inputenc}
\usepackage[english]{babel}
\usepackage{enumerate}
\usepackage{hyperref}

\usepackage[active]{srcltx}
\numberwithin{equation}{section}
\usepackage[T1]{fontenc}
\usepackage{color}
\newtheorem{theo}{Theorem}[section]
\newtheorem{lem}[theo]{Lemma}

\newtheorem{prop}[theo]{Proposition}

\theoremstyle{definition}

\newtheorem{defi}[theo]{Definition}
\newtheorem{rmk}{Remark}%[section]
\newtheorem{assump}{Assumption}[section]
\newcommand{\eps}{\varepsilon}

\newcommand{\N}{\mathbb{N}}
\newcommand{\R}{\mathbb{R}}
\newcommand{\Z}{\mathbb{Z}}
\newcommand{\Q}{\mathbb{Q}}

\DeclareMathOperator\supp{supp}
\def\1{\mathbbm{1}}

\let\mc=\mathcal
\def\F#1{\mc{F}_{e,#1}}

\newenvironment{formula}[1]{\begin{equation}\label{#1}}
{\end{equation}\noindent}

\def\Fi#1{\begin{formula}{#1}}
	\def\Ff{\end{formula}\noindent}

\setstcolor{red}

	\setstcolor{red}

\begin{document}
\title{\bf Pulsating solutions for multidimensional bistable and multistable equations}

\author{Thomas Giletti$^{\,\hbox{\small{a}}}$ \ and \ Luca Rossi$^{\,\hbox{\small{b}}}$\\
\\
\footnotesize{$^{\,\hbox{\small{a}}}$ Univ. Lorraine, IECL UMR 7502, Vandoeuvre-lès-Nancy, France} \\
\footnotesize{$^{\,\hbox{\small{b}}}$ CNRS, EHESS, PSL Research University, CAMS, Paris, France}}
\date{}

\maketitle

%----------------------------------------------------------------------

\abstract{We devote this paper to the issue of existence of pulsating travelling front solutions
 for spatially periodic heterogeneous reaction-diffusion equations in arbitrary dimension, in both bistable and more general multistable frameworks. 
% These are almost planar solutions which reduce to pulsating travelling
  %fronts in the monostable case.
 In the multistable case, the notion of a single front is not sufficient to understand 
 the dynamics of solutions, and
 we instead observe the appearance of a so-called propagating terrace. This roughly refers to a finite family of stacked fronts connecting intermediate stable steady states whose speeds are ordered. Surprisingly, for a given equation,
the shape of this terrace (i.e., the involved intermediate states or even the cardinality of the 
family of fronts) may depend on the direction of propagation.}

\section{Introduction}\label{sec:intro}

In this work we consider the reaction-diffusion equation
\Fi{eq:parabolic}
\partial_t u = \text{div} (A(x) \nabla u ) + f (x,u), \quad t \in \R, \ x \in \R^N   ,
\Ff
where $N \geq 1$ is the space dimension.
The diffusion matrix field $A = (A_{i,j})_{1 \leq i,j \leq N}$ is always assumed to be smooth 
and to satisfy the ellipticity condition
\Fi{eq:ellipticity}
\exists C_1, C_2 \in (0,\infty), \quad \forall x,\xi \in \R^N,
\quad C_1 |\xi|^2 \leq \sum_{i,j} A_{i,j} (x) \xi_i \xi_j \leq C_2 |\xi|^2.\Ff
As far as the regularity of the reaction term $f(x,u)$ is concerned, we assume that it is at least globally Lipschitz continuous (a stronger hypothesis will be made in the general multistable case; see below).

Equation~\eqref{eq:parabolic} is spatially heterogeneous. As our goal is to construct travelling fronts, 
i.e., self-similar propagating solutions, we impose a spatial structure on the heterogeneity. More precisely, we assume that the 
terms in the equation are all periodic in space, with the same period. 
For simplicity and without loss of generality up to some change of variables, we choose the periodicity cell to be $[0,1]^N$, that is,
\Fi{eq:periodicity}
\forall L \in \Z^N , \quad A (\cdot + L) \equiv A (\cdot) , 
 \ f(\cdot + L ,\cdot) \equiv f(\cdot, \cdot).
\Ff
From now on, when we say that a function is periodic, we always understand that
its period is $(1,\dots,1)$.

In the spatially periodic case, one can consider 
the notion of \textit{pulsating travelling front}, which we shall recall precisely below. Roughly, these are entire in time solutions which connect periodic steady states of the parabolic equation~\eqref{eq:parabolic}. The existence of such solutions is therefore deeply related to the underlying structure of \eqref{eq:parabolic} and its steady states.

In this paper, we shall always assume that \eqref{eq:parabolic} admits at least two spatially periodic steady states:
the constant~0 and a positive state~$\bar p(x)$. Namely, we assume that
$$f( \cdot, 0) 
 \equiv 0,$$
as well as
$$\left\{ \begin{array}{l}
\text{div} (A(x) \nabla \bar p) + f(x,\bar p) = 0, \vspace{3pt}\\
\forall L \in \Z^N, \quad \bar p(\cdot + L) \equiv \bar p >0.
\end{array}
\right.$$
We shall restrict ourselves to solutions~$u(t,x)$ of \eqref{eq:parabolic} that satisfy the inequality
$$0 \leq u \leq \bar p.$$
Notice that, as far the Cauchy problem is concerned, owing to the parabolic comparison principle, it is sufficient to assume that the above property is fulfilled by the initial datum (we restrict ourselves to bounded solutions, 
avoiding in this way situations where the comparison principle fails).
Let us also mention that 0 could be replaced by a spatially periodic steady state; we make this choice to keep the presentation simpler.

The steady states $0$ and $\bar p$ will be assumed to be asymptotically stable; we shall recall what this means in a moment. Then we distinguish the situation where these are the unique periodic steady states (bistable case)
to that where there is a finite number of intermediate stable states (multistable case).
In the latter, we will strengthen the stability condition.
\begin{assump}[Bistable case]\label{ass:bi}
	The functions $0$ and $\bar p$ are the unique asymptotically stable 
	periodic steady states of \eqref{eq:parabolic}.
	
	Furthermore, there does not exist any pair $q$, $\tilde q$ of periodic steady states of 
	\eqref{eq:parabolic} such that $0<q<\tilde q<\bar p$. 
\end{assump}
\begin{assump}[Multistable case]\label{ass:multi}
The function $\partial_u f(x,u) $ is well-defined and continuous. There is a finite number of asymptotically stable periodic steady states, among which~0 and~$\bar p$, 
and they are all linearly stable.

Furthermore,
for any pair of ordered periodic steady states $q<\tilde q$, there is a linearly stable periodic steady state $p$ such that $q \leq p \leq\tilde q$. 
\end{assump}
The main difference between these two assumptions is that only the latter allows the existence of intermediate stable steady states. 
As we shall see, the presence of such intermediate states might prevent
the existence of a pulsating travelling front connecting directly the two extremal steady states 0 and $\bar p$. More complicated dynamics involving a family of travelling fronts, which we refer to as a \textit{propagating terrace}, may instead occur.
We emphasize that the stable states in Assumption~\ref{ass:multi} are not necessarily ordered.

Let us recall the different notions of stability.
A steady state $p$ is said to be {\em asymptotically stable} if its
{\em basin of attraction} contains an open neighbourhood of $p$
in the $L^\infty(\R^N)$ topology; the basin of attraction of $p$
refers to the set of initial data for which the solution of the Cauchy problem
associated with~\eqref{eq:parabolic} converges uniformly to $p$
as $t \to +\infty$.

A periodic state $p$ is said to be {\em linearly stable} (resp.~{\em unstable}) if the linearized operator around $p$, i.e.,
$$\mathcal{L}_p w := \text{div} (A(x)\nabla  w ) + 
 \partial_u f (x, p(x)) w,$$
has a negative (resp.~positive)
principal eigenvalue in the space of periodic functions. Owing to the 
regularity of $f$ from Assumption~\ref{ass:multi}, it is
rather standard to construct sub- and supersolutions using the principal eigenfunction
and to use them to show that linear stability implies asymptotic stability. 
The converse is not true in general;
this is why the bistable Assumption~\ref{ass:bi} is not a particular case of Assumption~\ref{ass:multi}.

We also point out that
the second part of Assumption~\ref{ass:bi} automatically prevents the existence of intermediate
asymptotically stable steady states, thanks to a crucial result in dynamical systems due to Dancer and Hess~\cite{Tri}
known as ``order interval trichotomy''; see also~\cite{Matano}.
We recall such result in Theorem~\ref{DH} in the Appendix.

\begin{rmk}
In the case of the spatially-invariant equation
\Fi{eq:homo}
\partial_t u = \Delta u  + f (u), \quad t \in \R, \ x \in \R^N, 
\Ff
Assumption \ref{ass:bi} is fulfilled if and only if $\bar p$ is constant, say, equal to $1$,
and there exists $\theta\in(0,1)$ 
such that $f(u)<0$ for $u\in(0,\theta)$
and $f(u)>0$ for $u\in(\theta,1)$.
This is shown in Lemma \ref{lem:f0} below and the subsequent remark. 
Then, with the same arguments, one can readily check that
Assumption~\ref{ass:multi}
is equivalent to require that $f\in C^1([0,1])$ and that it has an odd (finite) number of zeroes 
such that, counting
from smallest to largest, the
odd ones (which include  $0$, $1$) satisfy $f'<0$ (these are the only stable periodic steady~states).
\end{rmk}

With a slight abuse of terminology, in the sequel we shall simply refer to the
asymptotic stability as ``stability''.
Then a solution will be said to be ``unstable'' if it is not (asymptotically)~stable.

\paragraph{The notion of pulsating fronts and terraces} Let us first recall the notion of \textit{pulsating travelling front}, which is the extension to the periodic framework of the usual notion of travelling front. We refer to~\cite{Xin-91} for an early introduction of this concept.

\begin{defi}\label{def:puls_front}
A {\em pulsating travelling front} for~\eqref{eq:parabolic} is an entire in time solution of the type
$$u(t,x) = U ( x , x\cdot e -ct),$$
where $c \in \mathbb{R}$,  $e \in \mathbb{S}^{N-1}$, the function $U(x,z)$ is periodic in the $x$-variable and satisfies
$$U (\cdot, - \infty) \equiv q_1 (\cdot) >  U(\cdot, \cdot)  > U (\cdot, +\infty) \equiv q_2 (\cdot).$$
Furthermore, we call $c$ the speed of the front, the vector $e$ its direction, and we say that $U$ connects 
$q_1$ to $q_2$.
\end{defi}
\begin{rmk}\label{rmk:trap}
The functions $q_1$, $q_2$ in the above definition are necessarily two steady states 
of \eqref{eq:parabolic}. 
Let us also point out that the change of variables $(t,x) \mapsto (x , x\cdot e - ct)$ is only invertible when $c \neq 0$, so that one should a priori carefully distinguish  both functions~$u$ and~$U$.
\end{rmk}

In the bistable case, our goal is to construct a pulsating front connecting 
$\bar p$ to $0$. Let us reclaim a few earlier results. In~\cite{Xin-91}, a pulsating front was already constructed in the special case where coefficients are close to constants. Yet dealing with more general heterogeneities turned out to be much more difficult, and only recently a pulsating front was constructed in~\cite{FZ} for the one-dimensional case, through an abstract framework 
which is similar to the one considered in the present work.
Higher dimensions were tackled in~\cite{Ducrot} under an additional nondegeneracy assumption and with a more PDE-oriented approach in the spirit of \cite{BH02}.

However, as mentioned before,
the notion of pulsating travelling front does not suffice to describe the dynamics 
in the more general multistable case. 
The good notion in such case is that of a propagating terrace, as defined in~\cite{DGM,GM}.
An earlier equivalent notion, called
 \textit{minimal decomposition}, was introduced in~\cite{FifMcL77} in
 the homogeneous case.
\begin{defi}\label{def:terrace}
A {\em propagating terrace} connecting $\bar p$ to $0$ 
in the direction $e \in \mathbb{S}^{N-1}$ is a couple of two finite sequences $(q_j)_{0 \leq j \leq J}$ and $(U_j)_{1 \leq j \leq J}$ such that:
\begin{itemize}
\item the functions $q_j$ are periodic steady states of \eqref{eq:parabolic} and satisfy
	$$\bar p \equiv q_0 > q_1 > \cdots > q_J \equiv 0;$$
\item for any $1 \leq j \leq J$, the function $U_j$ is a pulsating travelling front of 
\eqref{eq:parabolic} connecting $q_{j-1}$ to $q_j$ 
with speed $c_j \in \mathbb{R}$ and direction $e$;
\item the sequence $(c_j)_{1 \leq j \leq J}$ satisfies
	$$c_1 \leq c_2 \leq \cdots \leq c_J .$$
\end{itemize} 
\end{defi}

Roughly speaking, a propagating terrace is a superposition of pulsating travelling fronts 
spanning the whole range from $0$ to $\bar p$. 
We emphasize that the ordering of the speeds of the fronts involved in a propagating terrace is essential. 
Indeed, while there may exist many families of steady states and fronts satisfying the first two 
conditions in Definition \ref{def:terrace}, 
only terraces can be expected to describe the large-time behaviour of solutions of the 
Cauchy problem associated to \eqref{eq:parabolic}, see~\cite{DGM},
which makes them more meaningful.

\paragraph{Main results} Before stating our theorems, let us also recall a result by Weinberger.
\begin{theo}[Monostable case \cite{W02}]\label{thm:mono}
Let $p>q$ be two periodic steady states of \eqref{eq:parabolic}, and assume that 
any periodic function $u_0\in C(\R^N)$ satisfying $q\leq u_0\leq p$, $u_0\not\equiv q$, 
lies in the basin of attraction of $p$.
Then, for any $e \in \mathbb{S}^{N-1}$,
 there is some $c^*  \in \mathbb{R}$ such that
 a pulsating travelling front in the direction~$e$ with speed $c$ connecting $p$ to $q$
 exist if and only if $c \geq c^*$.
\end{theo}
Assumptions~\ref{ass:bi} or~\ref{ass:multi} allow us to apply this theorem 
around any given unstable periodic state~$q$ between $0$ and $\bar p$. 
To check the hypothesis of Theorem~\ref{thm:mono}, 
fix $x_0\in \R^N$ and let $p_+ >q$ be a stable state realizing the following minimum:
$$ \min\{p(x_0) : q<p<\bar p \ \ \mbox{and $p$ is a periodic stable state} \} .$$
Note that $p_+$ exists since we always assume that there is a finite number of  stable 
periodic steady states. By either Assumptions~\ref{ass:bi} of~\ref{ass:multi}, 
there does not exist any periodic steady state between $q$ and $p_+$. 
Because of this, and the stability of $p_+$, 
only the case~$(b)$ of the order interval trichotomy Theorem~\ref{DH} is allowed.
Namely, there exists a spatially periodic solution $u$ of \eqref{eq:parabolic} such that $u(k, \cdot)\to q$ as $k\to-\infty$ and $u (k,\cdot) \to p_+$ as $k \to +\infty$. By comparison principles, this implies that any periodic initial datum $q \leq u_0 \leq p_+$ with $u_0 \not \equiv q$ lies in the basin of attraction of 
$p_+$. We can therefore apply Theorem~\ref{thm:mono} and find a minimal 
speed $\overline{c}_q$ of fronts in a given direction 
$e\in\mathbb{S}^{N-1}$ connecting 
$p_{+}$ to~$q$. Applying the same arguments to \eqref{eq:parabolic} with
$f(x,u)$ replaced by $-f(x,-u)$,
we find a minimal speed $c '_q$ of fronts  $\tilde U$ in the direction $-e$ connecting $-p_{-}$ to $-q$, 
where~$p_{-}$ is the largest stable periodic steady state lying below~$q$.
Hence, $\underline{c}_q:=-c '_q$ is the maximal speed of fronts $U(x,z):=-\tilde U(x,-z)$ for~\eqref{eq:parabolic}
in the direction $e$ connecting~$q$ to~$p_{-}$.

After these considerations, we are in a position to state our last assumption.
\begin{assump}\label{ass:speeds}
For any unstable periodic steady state $q$ between $0$ and $\bar p$ and any $e\in\mathbb{S}^{N-1}$, there holds that
$$\overline{c}_q > \underline{c}_q,$$
where $\overline{c}_q$ and $\underline{c}_q$ are defined above.
\end{assump}
Notice that under the bistable Assumption~\ref{ass:bi}, clearly $p_{+} \equiv \bar p$ and $p_{-}\equiv0$. 
Therefore, in that case, Assumption~\ref{ass:speeds} means that pulsating fronts connecting $\bar p$ to
an intermediate state $q$ have to be strictly faster than pulsating fronts connecting~$q$ to~$0$. We point out that this hypothesis, though implicit, was already crucial in the earlier existence results for bistable pulsating fronts; see~\cite{Ducrot,FZ} where it was referred to as the \textit{counter-propagation} assumption.

When $u \mapsto f(x,u)$ is $C^1$, a sufficient condition ensuring Assumption~\ref{ass:speeds} is that $q$ is {\em linearly}
unstable.
In such a case there holds that $\overline{c}_q >0> \underline{c}_q$, as shown in
Proposition~\ref{prop:counter} in the Appendix.
We also show there for completeness that if $q$ is just unstable then
$\overline{c}_q \geq 0 \geq \underline{c}_q$. 
The fact that the minimal speed in a monostable problem cannot be 0 seems 
to be a natural property. Besides the non-degenerate ($q$ linearly unstable) case,
it is known to hold for homogeneous equations as well as
for some special (and more explicit) bistable equations, 
c.f.~\cite{DHZ,FZ} and the references therein.
However, as far as we know, it remains an open problem in general.

Our first main result concerns the bistable cases.
\begin{theo}[Bistable case]\label{th:bi}
If Assumptions~\ref{ass:bi} and~\ref{ass:speeds} are satisfied, then for any $e \in \mathbb{S}^{N-1}$, there exists a monotonic in time 
pulsating travelling front connecting $\bar p$ 
to $0$ in the direction $e$ with some speed $c (e) \in \mathbb{R}$.
\end{theo}
This theorem slightly improves the existence result of~\cite{Ducrot}, which 
additionally requires the stability or instability of the steady states to be linear. 
However, we emphasize that our argument is completely different: 
while in~\cite{Ducrot} the proof relies on an elliptic regularization technique, 
here we proceed through a time discretization and a dynamical system approach. 

\begin{rmk}
Our previous theorem includes the possibility of a front with zero speed.
However, there does not seem to be a unique definition of a pulsating front with zero speed in the 
literature, mainly because the change of variables $(t,x) \mapsto (x,x\cdot e-ct)$ 
is not invertible when $c=0$. 
Here, by Definition~\ref{def:puls_front} a front with zero speed is simply
a stationary solution $u(x)$ with asymptotics $u (x) - q_{1,2} \to 0$ as $x \cdot e \to \mp \infty$. As a matter of fact, in the zero speed case our approach provides the additional
property that there exists a function $U$ as in Definition~\ref{def:puls_front}, such that 
$u(t,x) = U (x, x \cdot e + z)$ 
solves~\eqref{eq:parabolic} for any $z \in \mathbb{R}$. However, this function $U$ lacks any regularity, so that in particular it is not a standing pulsating wave in the sense of~\cite{Ducrot}.
\end{rmk}

\begin{theo}[Multistable case]\label{th:multi}
If Assumptions~\ref{ass:multi} and~\ref{ass:speeds} are satisfied, then for any $e\in \mathbb{S}^{N-1}$, there exists a propagating terrace $((q_j)_j, (U_j)_j)$ connecting $\bar p$ 
to $0$ in the direction~$e$.

Furthermore, all the $q_j$ are stable steady states and all the fronts $U_j$ are monotonic in time.\end{theo}

Earlier existence results for propagating terraces dealt only with the one-dimensional case, where a Sturm-Liouville zero number and steepness argument is available~\cite{DGM,GM}. We also refer to \cite{Risler} where a similar phenomenon is studied by an energy method in the framework of systems with a gradient structure. As far as we know, this result is completely new in the heterogeneous and higher dimensional case.

The stability of these pulsating fronts and terraces will be the subject of a forthcoming work. Let us point out that, quite intriguingly, the shape of the terrace may vary depending on the direction. More precisely, for different choices of the vector $e$, the terrace may involve different intermediate states $(q_j)_j$; it is even possible that the number of such states varies, as we state in the next proposition.
\begin{prop}\label{prop:asymmetric}
There exists an equation~\eqref{eq:parabolic} in dimension $N=2$
for which Assumptions~\ref{ass:multi},~\ref{ass:speeds} hold and moreover:
\begin{itemize}
\item in the direction $(1,0)$, there exists a unique propagating terrace connecting $\bar p$ to 0, and it consists of exactly two travelling fronts;
\item in the direction $(0,1)$, there exists a unique propagating terrace connecting $\bar p$ to 0, and it consists of a single travelling front.
\end{itemize}
\end{prop} 
Uniqueness here is understood up to shifts in time of the fronts.
It will be especially interesting to study how this non-symmetric
phenomenon affects the large-time dynamics of solutions of the Cauchy~problem.

\paragraph{Plan of the paper}
We start in the next section with a sketch of our argument in the homogeneous case, to explain the main ingredients of our method. This relies on a time discretization, 
in the spirit of Weinberger \cite{W02}, 
and on the study of an associated notion of a discrete travelling front. 
For the sake of completeness, some of the arguments of~\cite{W02} will be reclaimed
along the proofs. We also point out here that the resulting discrete problem shares similarities 
with the abstract bistable framework considered in~\cite{FZ}, 
though we shall use a different method to tackle bistable and multistable equations 
without distinction.

The proof of the general case is carried out in several steps:
\begin{enumerate}
\item Introduction of the iterative scheme (Sections \ref{sec:discrete}, \ref{sec:ac1}).
\item Definition of the speed of the front (Section \ref{sec:c^*}).
\item Capturing the iteration at the good moment and position (Section \ref{sec:subsequence}).
\item Derivation of the travelling front properties (Section \ref{sec:uppermost}).
\end{enumerate}
At this stage we shall have constructed a {\em discrete} pulsating travelling front 
connecting~$\bar p$ to some stable periodic steady state~$0\leq p<\bar p$. 
In the bistable case, one necessarily has that $p\equiv0$ and then it only remains to prove that 
the front is actually a continuous front. 
For the multistable case, we shall iterate our construction getting a family 
of travelling fronts. In order to conclude that this is a propagating terrace, we need to show that 
their speeds are ordered; this is the only point which requires 
the linear stability in Assumption \ref{ass:multi}. 
Summing up, the method proceeds as~follows:
\begin{enumerate}
\setcounter{enumi}{4}
\item Construction of the (discrete) pulsating terrace (Section \ref{sec:terrace}).
\item Passing to the continuous limit (Section \ref{sec:continuous}).
\end{enumerate}
Finally, Section~\ref{sec:asymmetric} is dedicated to the proof of 
Proposition~\ref{prop:asymmetric},
which provides an example where the shape of the propagating terrace strongly
depends on its direction. To achieve this, we shall exhibit a bistable equation
for which pulsating fronts have different speeds, depending on their direction, see
Proposition \ref{pro:speeds} below.

\section{The 1-D homogeneous case}\label{sec:1D}

In order to illustrate our approach, let us consider the simpler (and, as far as travelling fronts are concerned, already well-understood~\cite{AW}) bistable homogeneous equation
\Fi{ref_frame}
\partial_t u = \partial_{xx} u + f(u),\quad t\in\R,\ x\in\R,
\Ff
with $f \in C^1 ([0,1])$  satisfying
$$f(0)=f(1)=0,\qquad f<0\ \text{ in }(0,\theta),\qquad f>0\ \text{ in }(\theta,1).$$
In this framework, pulsating fronts simply reduce to planar fronts, i.e., entire solutions
of the form $U(x-ct)$.

The hypotheses on $f$ guarantee that Assumption \ref{ass:bi} is fulfilled with $\bar p\equiv1$.
They also entail the ``counter-propagation'' property, Assumption \ref{ass:speeds},
because in the homogeneous monostable case travelling fronts have positive speeds, see \cite{AW}. 
Namely, fronts connecting $1$ to $\theta$ exist for speeds $c$ larger than some $\overline c>0$,
whereas fronts connecting~$\theta$ to~$0$ exist for speeds $c$ smaller than some $\underline c<0$
(the latter property  is derived from~\cite{AW} by considering fronts moving leftward
for the equation for $\theta-u$).

The equation in the frame moving rightward with speed $c\in\R$ reads
\Fi{moving}
\partial_t u = \partial_{xx} u + c \partial_x u+f(u),\quad t\in\R,\ x\in\R.
\Ff

%----------------------------------------------------------------------

\subsection{The dynamical system}

We start by placing ourselves in a more abstract framework which we shall use to define a candidate front speed $c^*$, in the same way as in~\cite{W02}. We shall then turn to the construction of a travelling front connecting 1 to 0. We point out that in~\cite{W02}, such a travelling front was only shown to exist in the monostable case, and that a different argument is needed to deal with bistable or more complicated situations.
		
For any given $c\in\R$, we call $\mathcal{F}_c$
the evolution operator after time 1 associated with~\eqref{moving}.
Namely, $\mathcal{F}_c [ \phi ](x):=v(1,x)$, where 
$v$ is the solution of \eqref{moving} emerging from
the initial datum $v(0,x)=\phi(x)$. It follows from the
parabolic strong maximum principle that the operator~$\mathcal{F}_c$ is increasing.

Let us already point out that the profile $U$ of a usual travelling front
$U(x-ct)$ for~\eqref{ref_frame} 
is a stationary solution of \eqref{moving} and thus
a fixed point for the operator $\mathcal{F}_c$. As a matter of fact, in the homogeneous case the converse is also true (this follows for instance from a uniqueness result 
for almost planar fronts derived in \cite{BH12}). 
Therefore, our goal in this section will be to construct such a fixed point. \\

Consider a function $\phi\in W^{1,\infty}(\R)$ satisfying
\Fi{phi}
\phi \ \text{ is nonincreasing},\qquad
\phi(-\infty)\in(\theta,1),\qquad
\phi= 0\ \text{ in }\ [0,+\infty).
\Ff
We then define a sequence $(a_{c,n})_{n\in\N}$ through the following iterative procedure:
$$a_{c,0}:= \phi,$$
$$a_{c,n+1} := \max \{ \phi, \mathcal{F}_c [a_{c,n}]\},$$
where the maximum is taken at each $x\in\R$.

It follows from the monotonicity of $\phi$ and $\mathcal{F}_c$ (the latter being strict) that
$a_{c,n}(x)$ is nondecreasing with respect to $n$ and 
nonincreasing with respect to $x$,
and that it satisfies $0<a_{c,n}<1$.
Then, observing that 
\Fi{E0}
\mathcal{F}_c [V]=\mathcal{F}_0 [V] (\cdot+c),
\Ff
for any function $V$, we deduce that $a_{c,n}$ is nonincreasing with respect to $c$.
One also checks by iteration that $a_{c,n}(+\infty)=0$, thanks to standard parabolic arguments.	
All these properties are summarized in the following.
\begin{lem}\label{lem:acn}
	The sequence $(a_{c,n})_{n\in\N}$ is nondecreasing and satisfies $0<a_{c,n}<1$ 
	and $a_{c,n}(+\infty)=0$ for all $n\geq1$. 
Moreover, $a_{c,n}(x)$ is nonincreasing with respect to both $c$ and $x$, 
the latter monotonicity being strict in the set where $a_{c,n}>\phi$.
\end{lem}

Lemma \ref{lem:acn} implies that $(a_{c,n})_{n\in\N}$ converges pointwise to some 
nonincreasing function $\phi\leq a_c\leq 1$. 
The convergence actually holds locally uniformly in~$\R$, because the~$a_{c,n}$ are 
equi-uniformly Lipschitz-continuous, due to parabolic estimates. 
We also know that the $a_c$ are nonincreasing with respect to $c$.

We then introduce
$$c^* := \sup \{ c \in \R \, : \ a_c \equiv 1 \}.$$
One may check that $c^*$ is indeed a well-defined real number. 
Without going into the details (this a particular case of either Section~\ref{sec:N} or~\cite{W02}), we simply point out that this can be proved using some super- and subsolutions which exist
thanks to the Lipschitz continuity of $f$
as well as to the choice of $\phi (-\infty)$ in the basin of attraction of $1$.

We further see that the definition of $c^*$ does not depend on the particular choice of the 
initialising function $\phi$. Indeed, if $\tilde\phi$ satisfying \eqref{phi} is the initialisation of another sequence,
then for $c<c^*$ there holds that $a_{c,n}>\tilde\phi$ for $n$ sufficiently large.
From this and the monotonicity of $\mathcal{F}_c$
one deduces by iteration that the value of $c^*$ obtained starting from 
$\phi$ is larger than or equal to the one provided by $\tilde\phi$.
Equality follows by exchanging the roles of $\phi$ and $\tilde\phi$.

We shall also use the fact that 
\begin{equation}\label{c*not1}
a_{c^*} \not \equiv 1.
\end{equation}
This comes from the openness of the set $\{c \in \R  : a_c \equiv 1 \}$, which is established in either 
Section~\ref{sec:N} or~\cite{W02} in the more general periodic case. 
Let us briefly sketch a more direct proof. Let $c\in\R$ be such that $a_{c} \equiv 1$. 
We can find $\bar n$ such that $a_{c,\bar n}(1)>\phi(-\infty)$.
Arguing by induction and exploiting \eqref{E0}, one sees that
$$\forall\delta>0,\ n\in\N,\  x\in\R,\quad
a_{c+\delta,n}(x)\geq a_{c,n}(x+n\delta).$$
Thus, $a_{c+\frac1{\bar n},\bar n}(0)>\phi(-\infty)$ which implies that
$a_{c+\frac1{\bar n},\bar n}>\phi$
because $a_{c+\frac1{\bar n},\bar n}$ and $\phi$ are
nonincreasing and $\phi$ is supported in $(-\infty,0]$.
Using the next result we eventually deduce that 
$a_{c''} \equiv 1$ for all $c''$ in some neighborhood of $c$, and thus $c^* > c$.
\begin{lem}\label{lem:c*open}
	Let $c'\in\R$ and $\bar n\in\N$ be such that $a_{c',\bar n}>\phi$. 
	Then $a_{c''}\equiv1$ for all~$c''<c'$.
\end{lem}

\begin{proof}
	The monotonicities provided by Lemma~\ref{lem:acn}
	yield $a_{c'',\bar n+m} > \phi$ for all $c''\leq c'$ and $m\in\N$, which,
	recalling the definition of the sequences 
	$(a_{c,n})_{n\in\N}$, implies in turn that
	$a_{c'',\bar n+m}=(\mathcal{F}_{c''})^m [a_{c'' ,\bar n}]$.
	Then, taking $c''<c'$ and exploiting \eqref{E0}, we get
	\[\begin{split}
	\forall m\in\N,\ x\in\R,\quad
	a_{c'',\bar n+m}(x) & = (\mathcal{F}_{c''})^m [a_{c'', \bar n}] (x) \\
	&=(\mathcal{F}_{c'})^m [a_{c'' ,\bar n}] (x- (c'-c'') m)\\
	&\geq
	(\mathcal{F}_{c'})^m [a_{c',\bar n}] (x - (c'-c'') m)\\
	&=a_{c',\bar n+m}(x- (c'-c'') m).
	\end{split}\]
Passing to the limit as $m \to +\infty$ (and using again the monotonicity of the sequence) 
we find that $a_{c''}(x)\geq a_{c',n}(-\infty)$
for all $x\in\R$ and $n\in\N$.
Observe that $(a_{c',n}(-\infty))_{n\in\N}$ is the solution of the ODE $U'=f(U)$ computed on the integers and starting from
$\phi(-\infty)>\theta$, whence it converges to $1$.
This shows that $a_{c''} \equiv 1$.
\end{proof}

%----------------------------------------------------------------------

\subsection{Capturing the sequence at the good moment and position}

From here we diverge from
Weinberger's scheme which, as we mentioned above,
 does provide a front in the monostable case but not
in the bistable one. 

Consider $c < c^*$. Because $a_c \equiv 1$,
we have seen before that we can find $n(c)$ such that
$a_{c,n(c)+m} > \phi$ for $m\in\N$. This means that, starting from $n(c)$,
the sequence $(a_{c,n})_{n\in\N}$ is simply given by iterations of $\mathcal{F}_c$, namely,
\Fi{nc1}
\forall m\in\N,\quad
a_{c,n(c)+m}=(\mathcal{F}_c)^m [a_{c,n(c)}].
\Ff
Fix $\theta'\in(\theta,\phi(-\infty))$ and, for $n\geq n(c)$, define the point 
$x(c,n)$ through the relation
$$a_{c, n} (x (c,n)) =\theta'.$$
Note that $x(c,n)$ exists because $ a_{c,n}(-\infty)\geq\phi(-\infty)>\theta'$ 
and $a_{c,n}(+\infty)=0$ by Lemma~\ref{lem:acn}.
Moreover we claim that, by construction of $c^*$, there holds that
\begin{equation}\label{claim1}
\limsup_{n \to \infty} \frac{x (c,n)}{n} \leq c^* - c.
\end{equation}
Let us postpone the proof of this for a moment and continue with our construction.
By \eqref{claim1}, one readily sees that, up to increasing $n (c)$ if need be,
the following holds:
\Fi{nc2}
\forall 0\leq m\leq 1/\sqrt{c^*-c},\quad
x (c,  n (c)+m  ) - x (c,   n(c)) \leq 2\sqrt{c^* -c}.
\Ff 
Conditions \eqref{nc1},\eqref{nc2} determine our choice of the diagonal sequence
$(a_{c,n(c)})_{c<c^*}$.

Let $u_c (t,x)$ denote the solution of the Cauchy problem for
\eqref{moving} with initial datum~$a_{c,n(c)}$ (notice that $u_c (t,x)$ satisfies parabolic estimates up to time $t=0$ because $a_{c,n(c)}=\mathcal{F}_c[a_{c,n(c)-1}]$). Property \eqref{nc1} and the monotonicity of $(a_{c,n})_{n\in\N}$ imply that
$$\forall n\in\N,\quad
u_c(n+1,\cdot)\equiv a_{c,n(c)+n+1}\geq a_{c,n(c)+n}\equiv u_c(n,\cdot),$$
that is, the sequence $(u_c(n,\cdot))_{n\in\N}$ is nondecreasing. 
Furthermore, the function $u_c$ inherits the monotonicity
in $x$ of the initial datum, which is strict by Lemma~\ref{lem:acn}
because $ a_{c,n}>\phi$.
 
We finally consider the translation $u_c(t,x+x(c,  n(c)))$ of $u_c$. 
By parabolic estimates up to $t=0$, we have that (up to subsequences) 
$$u_c (t  ,x + x(c, n(c))) \to a^* (t,x)\quad\text{ as } c \nearrow c^*,$$
locally uniformly in $(t,x)\in [0,+\infty) \times\R$, where~$a^*(t,x)$ satisfies the equation \eqref{moving} with $c=c^*$. 
We further know that 
$$a^*(0,0)=\theta'$$ 
and that $a^*(n,x)$ is nondecreasing in~$n\in\N$ and 
nonincreasing in~$x\in\R$.

Let us now prove~\eqref{claim1}. First, the function $\phi $
being nonincreasing, for any $c < c^*$ we deduce from \eqref{E0} that 
$$a_{c,1} = \max \{ \phi, \mc{F}_c [\phi] \} \leq 
\max \{ \phi (\cdot + (c-c^*)), \mc{F}_{c^*} [\phi] (\cdot + (c-c^*)) \} = a_{c^*,1} (\cdot + (c-c^*)).$$
An iterative argument then shows that
\begin{equation}\label{claim22}
\forall n\in\N,\quad a_{c,n} \leq a_{c^*,n} (\cdot + n (c-c^*)).
\end{equation}
Now it follows from \eqref{c*not1} that $\inf a_{c^*} \leq \theta $. Indeed, assume by contradiction that $\inf a_{c^*} > \theta$. Then by comparison with the ODE, we immediately conclude that $(\mathcal{F}_c)^m [a_{c^*}] \to 1$ as $m \to +\infty$. However, by construction, $a_{c^*,n+1} \geq \mathcal{F}_{c^*} [a_{c^*,n}]$, hence $a_{c^*} \geq \mathcal{F}_{c^*} [a_{c^*}]$ and therefore
the monotonicity of $\mc{F}_{c^*}$ eventually yields 
$$a_{c^*}\geq\lim_{m\to+\infty}(\mathcal{F}_c)^m [a_{c^*}]=1,$$
contradicting \eqref{c*not1}.
We infer from the above that there exists $X_{\theta'} \in\R$ such that
$$\forall  n \in \mathbb{N},\quad
 \theta'>a_{c^*} (X_{\theta'})\geq a_{c^*,n} (X_{\theta'})\geq a_{c,n} (X_{\theta'}+n (c^* -c)),$$
where the last inequality follows from~\eqref{claim22}. This means that
$$\forall  n \in \mathbb{N},\quad
x (c,n)<X_{\theta'} + n (c^* -c),$$
from which \eqref{claim1} immediately follows.

%----------------------------------------------------------------------

\subsection{The function $a^*$ converges to the profile of a front}

We recall that, by construction, the sequence $a^* (n, \cdot)$ is nondecreasing with respect to $n \in \mathbb{N}$. In particular, we can define
$$U^* (t,x) := \lim_{n \to +\infty} a^* (t+n,x),$$
By parabolic estimates, the above limit exists (up to subsequences)
locally uniformly in $(t,x)\in\R^2$ and
$U^*$ is a periodic in time solution of \eqref{moving} with $c = c^*$.
Moreover, $U^*$ satisfies $U^*(0,0)\geq\theta'$
and inherits from $a^*$ that it is nonincreasing with respect to $x$. Let us check that it is actually a travelling front.

Using parabolic estimates and the monotonicity with respect to $x$, we see that the sequences $(U^*(t,x\pm n))_{n\in\N}$ converge locally uniformly in $(t,x)\in \R^2$ (up to subsequences)
to two steady states $U^*_\pm$ of the same ODE \,$U'=f(U)$ (here we used that this
ODE does not admit non-trivial periodic solutions), i.e., $U^*_\pm$ are constantly equal 
to~$0$,~$\theta$ or $1$.
The fact that $U^*(0,0)\geq\theta'>\theta$ and the monotonicity in $x$ then imply that
$$U^*_- = U^* (\cdot, -\infty) \equiv 1.$$

Next, we claim that $U^*_- \equiv0$. 
Once this claim is proved, 
one may show by a sliding argument as in~\cite{BH12} that $U^*$ 
is actually independent of $t$, and thus 
it is the profile of a front moving with speed $c^*$. Therefore, in order to conclude this preliminary section, we need to rule out the cases $U^*_+\equiv  \theta$ and 
$U^*_+\equiv  1$.
Condition \eqref{nc2}
is specifically devised to prevent the latter possibility. Indeed, it 
yields
$$\forall 0\leq m\leq 1/\sqrt{c^*-c},\quad
u_c (m, x (c,  n(c)) + 2\sqrt{c^* -c} ) \leq 
u_c (m, x (c,  n(c)+m))=\theta'.$$ 
Passing to the limit as $c \nearrow c^*$ in this inequality we get 
\begin{equation*}
\forall m \in \mathbb{N},\quad a^* (m, 0) \leq \theta',
\end{equation*}
whence $U^* (0,0)\leq\theta'$. By the monotonicity in $x$, we then derive
$$U^*_+ = U^* (\cdot, +\infty)<1.$$
It remains to rule out the case $U^*_+ \equiv\theta$.
To achieve this, we shall compare $c^*$ with the spreading 
speeds associated with the restrictions of $f$ to $[0,\theta]$ and $[\theta,1]$ respectively, 
which are of the well-known (even in the periodic and multidimensional case) monostable type.
This is where the ``counter-propagation'' property comes into play.
We recall that such a property is guaranteed in the homogeneous case we are considering now, 
but should be imposed in general through Assumption \ref{ass:speeds}.

We proceed by contradiction and suppose that $U^*_+ \equiv \theta$. 
Thus $U^* (0,\cdot) \geq \theta$, as well as $U^* (0,\cdot) \geq \overline{u}_0$ defined by
$$\overline{u}_0 = \theta' \, \1_{(-\infty,0]}  + \theta \,\1_{(0,+\infty)} .$$
Consider now the solution $\overline{u}$ of \eqref{ref_frame} with initial datum $\overline{u}_0$.  Since $\theta$ is an unstable steady state, we can use the well-known result about the spreading speed for solutions of the monostable equation from~\cite{AW}. 
Namely, we find a speed $\overline{c}>0$ such that 
$$\forall c < \overline{c}, \quad \overline{u} (t,ct) \to 1 \quad\text{as }\;t \to +\infty,$$
$$\forall c > \overline{c}, \quad \overline{u} (t,ct) \to \theta \quad\text{as }\;t \to +\infty.$$
It is also proved in~\cite{AW} that $\overline{c}$ coincides with the minimal speed of fronts,
c.f.~Theorem~\ref{thm:mono}, that is,
using the same notation as in the introduction, there holds that~$\overline{c} = \overline{c}_\theta$.
Since $U^* (t,x-c^*t)$ satisfies \eqref{ref_frame} and 
$U^* (0,\cdot)\geq \overline{u}_0$, we infer by comparison that for all~$c< \overline{c}$,
there holds
$U^* (t,(c-c^*)t) \to 1  $ as $t \to +\infty$. Recalling that $U^*$ is 
periodic in time and that we are assuming that $U^*_+ \equiv \theta$, 
we eventually find that 
$c^* \geq \overline{c} >0$.

Let us go back now to the construction of $a^*$, $U^*$. We have that, up to a subsequence,
$$U^* (0,x) = \lim_{k \to +\infty} \Big(\lim_{c  \nearrow c ^*} a_{c, n(c) +k} \big( x + x(c,n(c))\big)\Big).$$
In particular, one can take a sequence $c_k \nearrow c^*$ such that, locally uniformly in $x$, 
\begin{equation}\label{U*0}
U^* (0,x) = \lim_{k \to +\infty} a_{c_k,n(c_k)+k} (x + x(c_k, n(c_k))).
\end{equation}
Now for any $c<c^*$ and $n\in\N$, 
let $x ' (c,n)$ be such that
$$a_{c,n} ( x' (c,n)) = \frac{\theta}{2} .$$
Let us extract another subsequence so that the solution of \eqref{moving} with initial datum 
$$a_{c_k, n(c_k)+k} (x+x'(c_k, n(c_k)+k))$$
converges locally uniformly in $(t,x)\in\R^2$
to some $V^* (t,x)$,
 which is an entire solution of \eqref{moving} with $c=c^*$. Moreover, $V^* (n,x)$ is nondecreasing in $n \in \mathbb{Z}$, nonincreasing in $x \in \mathbb{R}$, and satisfies $V^* (0,0 ) =  \theta/2 .$
One can further see that $V^* (0,\cdot) \leq \theta$; this follows from the fact that
$x' (c_k, n(c_k)+k) - x (c_k, n(c_k)) \to +\infty$, which, in turn, is a consequence of~\eqref{U*0}
and of the contradictory assumption $U^*_+ \equiv \theta$.
In particular, we have that $V^* (0,\cdot) \leq \underline{u}_0$ defined by
$$\underline{u}_0 = \theta \, \1_{(-\infty,0]}  + \frac{\theta}{2}  \,\1_{(0,+\infty)} .$$
Owing again to the spreading result for the monostable equation, 
there exists a speed $\underline{c} <0$ such that
the solution $\underline{u}$ of \eqref{ref_frame} 
emerging from $\underline{u}_0$ satisfies
$$\forall c < \underline{c}, \quad \underline{u} (t,ct) \to \theta\quad\text{as }\;t \to +\infty,$$
$$\forall c > \underline{c}, \quad \underline{u} (t,ct) \to 0\quad\text{as }\;t \to +\infty.$$
On one hand,
by comparison we get that $V^* (t,x - c^* t) \leq \underline{u} (t,x)$. On the other hand, by monotonicity we know that $V^* (n,x) \geq \frac{\theta}{2}$ for all $n \in \mathbb{N}$,
$x\leq0$. One then easily infers that $c^* \leq \underline{c} <0$. We have finally reached a contradiction.

\section{The iterative scheme in the periodic, $N$-dimen\-sional case}\label{sec:N}

We now turn to the general periodic case in arbitrary dimension.
Because the equation is no longer invariant by any space translation, we need to introduce a more complicated operator involving also a somewhat artificial variable. This makes things more technical, though the overall strategy remains the same.

\subsection{A time discretization}\label{sec:discrete}

The main ingredient of our proofs is inspired by Weinberger \cite{W02}, and consists in looking for travelling fronts as fixed points of an appropriate family of mappings issued from a time discretization of \eqref{eq:parabolic}.

First, we use the notation
$$v(t,y;x\mapsto v_0(x))$$
to indicate the solution to \eqref{eq:parabolic} with initial datum $v_0$, evaluated at $(t,y)$. In the sequel, we shall often 
omit to write ``$x\mapsto$'' and we shall just use $x$ as the variable involved in the initial datum. 

Let us now recall (see Definition~\ref{def:puls_front}) that a \textit{pulsating travelling front} in a direction $e \in \mathbb{S}^{N-1}$ is a solution of \eqref{eq:parabolic} of the form
$$u(t,x) = U (x,x\cdot e - ct)$$
with $U (x,z)$ periodic in the $x$-variable and converging to two distinct steady states as~$z \to \pm \infty$. 
In particular, one may look at a travelling front as a family $(U(x,z))_{z\in\R}$, using the second variable as an index.

Let us translate the notion of pulsating travelling front to the discrete setting.
\begin{defi}\label{def:discrete_front}
A {\em discrete travelling front} in a direction $e \in \mathbb{S}^{N-1}$ with speed $c \in \mathbb{R}$ is a function $U(y,z)$ which is periodic in its first variable, 
satisfies
$$\forall(y,z)\in\R^{N+1},\quad
v(1,y; x \mapsto U(x,z + x \cdot e)) \equiv U(y,z + y \cdot e - c),$$
and connects two steady states $q_1$ and $q_2$, i.e.,
$$U(\cdot, - \infty) \equiv q_1 (\cdot) > U (\cdot, \cdot) > U (\cdot , +\infty) \equiv q_2 (\cdot),$$
where convergences are understood to be uniform.\end{defi}

Clearly, if $u(t,x) = U (x,x\cdot e - ct)$ is a (continuous) 
pulsating travelling front then~$U(x,z)$ is a discrete travelling front, at least if $c \neq 0$ so that the change of variables $(t,x) \mapsto (x, x\cdot e - ct)$ is invertible. The converse is a priori not obvious: we immediately deduce from 
Definition~\ref{def:discrete_front} that, for every $\tau \in \mathbb{R}$,  
the function $U(x,x\cdot e - ct)$ coincides with a solution $u_\tau$ of the parabolic 
equation~\eqref{eq:parabolic} on the 1-time-step set
$(\{\tau\} + \mathbb{Z} ) \times \mathbb{R}^N$, but to recover a pulsating front we
should have that the $u_\tau$ are time-translations of the same solution. 
This difficulty will be overcome by instead considering different discretizations 
with time steps converging to $0$.

\begin{rmk}
This part of the argument, about going from discrete to continuous travelling fronts, was actually omitted by Weinberger in 
the paper \cite{W02}
that we refer to in Theorem~\ref{thm:mono} above. 
A proof in the homogeneous case can be found in~\cite{LWL}. However this does not seem to raise significant difficulties in the periodic case.
Let us also mention 
that one can see that a discrete travelling front gives rise to an
``almost planar generalized transition front'' in the sense of Berestycki and Hamel \cite{BH12}.
Then, in some situations (typically under some strong stability assumptions and provided also that the front speed is not zero), it is shown 
in \cite[Theorem 1.14]{BH12} that an almost planar transition front is also a travelling front in a usual sense.
\end{rmk}

Definition~\ref{def:discrete_front} leads us to define the 
family of mappings $\mathcal{F}_{e,c}:
L^\infty(\R^{N+1})\to L^\infty(\R^{N+1})$
for $e\in \mathbb{S}^{N-1}$ and $c\in\R$ as follows:
\Fi{eq:def_mapping}
\mathcal{F}_{e,c}[V](y,z):=v(1,y;V (x , z+ x \cdot e  - y \cdot e + c)).
\Ff
Rewriting the mapping $\mathcal{F}_{e,c}$ as
\Fi{eq:def_mapping1}
\mathcal{F}_{e,c}[V](y,z+y\cdot e-c)=v(1,y;V (x , z + x \cdot e)),
\Ff
we see that the discrete travelling fronts are given by the fixed points of 
$\mathcal{F}_{e,c}$. Formula~\eqref{eq:def_mapping1} also allows one to use parabolic estimates to obtain regularity with respect to~$y\mapsto(y,z+y\cdot e)$.

In a similar fashion, notice that any spatially periodic stationary state $p(y)$ of \eqref{eq:parabolic} is a $z$-independent fixed point of $\mathcal{F}_{e,c}$ for any $c$ and $e$. The converse is also true, as a consequence of the next result.
\begin{prop}\label{prop:energy1}
Let $u(t,x)$ be a 1-periodic in time solution of \eqref{eq:parabolic} which is also periodic in space.

Then $u$ is actually stationary in time.
\end{prop}
\begin{proof}
Let us first introduce the energy 
$$E (w): = \int_{[0,1]^N} \left(\frac{A |\nabla w |^2}{2} - F (x,w) \right)dx,$$
for any periodic function $w\in C^1(\R^N)$, where
$$F (x,s) := \int_0^s f(x,\sigma) d\sigma.$$
Then one may check that the solution $u(t,x)$ of \eqref{eq:parabolic} satisfies
$$\partial_t E (u (t,\cdot)) =  - \int_{[0,1]^N} | \partial_t u |^2dx\leq0 .$$
On the other hand, the mapping $t\mapsto E (u (t,\cdot))$ is $1$-periodic, whence 
it is necessarily constant. This implies that 
$\partial_t u \equiv 0$.
\end{proof}

We also derive several properties of the mapping 
$\mathcal{F}_{e,c}$ which will be useful later.
\begin{prop}\label{prop:Fec}
For given $e \in \mathbb{S}^{N-1}$ and $c \in \mathbb{R}$, the mapping $\mathcal{F}_{e,c}$ satisfies the following properties.
\begin{enumerate}[$(i)$]
\item {\em Periodicity:} if $V (y,z)$ is periodic with respect to $y \in \R^N$ then
this holds true for
$\mathcal{F}_{e,c} [V] (y,z)$.

\item {\em Monotonicity:} 
if $V_1 \leq V_2$ then $$\mathcal{F}_{e,c} [V_1] \leq \mathcal{F}_{e,c} [V_2];$$
if in addition $\sup_{y\in\R^N}(V_2-V_1)(y,z+y\cdot e)>0$ for all $z\in\R$, then 
$$\mathcal{F}_{e,c} [V_1] < \mathcal{F}_{e,c} [V_2].$$

\item {\em Continuity:}  if $V_n (y,z+ y \cdot e)\to V_\infty (y,z+ y \cdot e)$
as $n\to+\infty$ locally uniformly in~$y\in\R^N$, for some $z\in\R$, then
$$\F{c}[V_n] (y,z+ y \cdot e-c)\to \F{c}[V_\infty] (y,z+ y \cdot e-c)
\quad\text{as }\;n\to+\infty$$
locally uniformly in $y\in\R^N$.
	
\item {\em Compactness:} for any sequence $(V_n)_{n \in\mathbb{N}}$ 
bounded in $L^\infty(\R^{N+1})$ and any $z \in \mathbb{R}$, there exists a subsequence (depending on $z$) along which the function $y \mapsto \mathcal{F}_{e,c} [V_n] (y,z+y\cdot e)$ converges in $L^\infty_{loc}(\R^N)$ as $n\to+\infty$.
\end{enumerate}
\end{prop}
\begin{proof}
Let $V(y,z)$ be a periodic function in its first variable. Then for any $y \in \R^N$, $z\in \mathbb{R}$ and $L \in \Z^N$, the periodicity of equation \eqref{eq:parabolic} yields
\[\begin{split}
\F{c}[V](y+L,z) &=v(1,y+L;V (x , z+ x \cdot e  - y \cdot e - L\cdot e + c))\\
&=v(1,y;V (x+L , z+ x \cdot e  - y \cdot e + c))\\
&=\F{c}[V](y,z).
\end{split}\]
This proves $(i)$.

Statement $(ii)$ simply follows from \eqref{eq:def_mapping1} and the parabolic weak and strong comparison principles.

The continuity property follows from standard parabolic estimates.
Indeed, take a sequence $(V_n (y,z+ y\cdot e))_{n \in\mathbb{N}}$  converging locally uniformly in $y$ and for some 
$z\in\R$ to $V_\infty (y,z+ y\cdot e)$. Then the functions $(w_n)_{n \in\mathbb{N}}$ defined~by
$$w_n (t,y) := v (t,y;V_n (x,z+ x \cdot e)) - v (t,y; V_\infty (x,z+x \cdot e))$$
solve, for any fixed $z \in \mathbb{R}$, a linear parabolic equation of the type
$$\partial_t w_n =\text{div} (A(y) \nabla w_n) +  g^{z,n} (t,y) w_n,$$
with $|g^{z,n}|$ less than or equal to the Lipschitz constant of $f$,
together with the initial condition $V_n (x, z + x \cdot e) - V_\infty (x,z+x \cdot e)$. 
It follows from the comparison principle and parabolic estimates
that $(w_n)_{n\in\N}$ converges to 0 locally uniformly with respect to~$t>0$, $y \in \mathbb{R}^N$ 
. In particular, $y\mapsto v(1,y;V_n(x,z+x \cdot e))$ converges locally uniformly as $n \to +\infty$ to $v(1,y;V_\infty (x,z + x \cdot e))$,
which owing to \eqref{eq:def_mapping1} translates into the desired property.

The last statement $(iv)$ is an immediate consequence of the parabolic estimates.
\end{proof}

Let us point out that the operators 
$\mathcal{F}_{e,c}$ were initially introduced by Weinberger in \cite{W02}, 
who exhibited the existence of a spreading speed of solutions in a rather general context, but only proved the existence of pulsating fronts in the monostable case. These operators also fall into the scope of~\cite{FZ} (though they 
lack the compactness property required in some of their results). 
In particular, though one may proceed as in the aforementioned paper at least in the bistable case, we suggest here a slightly different approach. In some sense, our method is actually closer to the initial argument of Weinberger in~\cite{W02}, and though we do not address this issue here, it also seems well-suited to check that the speed of the pulsating front (or the speeds of the propagating terrace) also determines the spreading speed of solutions 
of the Cauchy problem associated with~\eqref{eq:parabolic}.

\subsection{Basic properties of the iterative scheme}\label{sec:ac1}

From this point until the end of Section~\ref{sec:discreteTF}, we assume that the following holds.
\begin{assump}\label{ass:mix}
The equation \eqref{eq:parabolic} admits a finite number of asymptotically stable steady states,
among which $0$ and $\bar p$.

Furthermore, for any pair of ordered periodic steady states $q < \tilde{q}$, 
there is an asymptotically stable steady state $p$ such that $q \leq p \leq \tilde{q}$.
\end{assump}

This hypothesis is guaranteed by both the bistable Assumption~\ref{ass:bi} and the multistable Assumption~\ref{ass:multi}.

For the sake of completeness as well as for convenience (several of the following properties will play in important role here), we repeat some of the arguments of~\cite{W02}. In particular, we start by reproducing how to define the speed $c^*$ (depending on the direction $e\in \mathbb{S}^{N-1}$) which was shown in~\cite{W02} to be the spreading speed for planar like solutions of the Cauchy problem. Roughly, for any $c<c^*$ we construct a time increasing solution of the parabolic equation in the moving frame with speed $c$ in the direction $e$. Later we shall turn to a new construction of a pulsating travelling front connecting $\bar p$ to a stable periodic steady state $p <\bar p$ with~speed~$c^*$.

The construction starts with an $L^\infty$ function $\phi$ satisfying the following:
\Fi{ass:phi}
\begin{cases}
\displaystyle
\phi (y,z) \mbox{ is periodic in $y \in \R^N$, and nonincreasing in $z \in \R$},\vspace{5pt}\\
\phi (y,z) \mbox{ is uniformly continuous in $(y,z)\in\R^{N+1}$}, 
\vspace{5pt}\\
\displaystyle\phi(y,z)=0 \mbox{ for } y\in\R^N,\ z\geq0 , \vspace{5pt}\\
\displaystyle\phi(y,-\infty)<\bar p(y),\vspace{5pt}\\
\displaystyle\exists\delta>0\ \text{ such that } \phi(y,-\infty)-\delta\ \text{ lies
	in the basin of attraction of $\bar p$.}
\end{cases}
\Ff
Observe that the limit $\phi(y,-\infty)$ exists uniformly with respect to $y$, and thus it is continuous (and periodic).
The last condition is possible due to the (asymptotic) stability of 
$\bar p$. Owing to the comparison principle, it implies that $\phi(y,-\infty)$ lies
in the basin of attraction of $\bar p$ too.

Then, for any $e \in \mathbb{S}^{N-1}$ and $c \in \mathbb{R}$, we define the sequence $(a_{c,n})_{n\in\N}$~by
\Fi{eq:def_acn}
\begin{array}{c} a_{c,0} := \phi,\vspace{5pt}\\
a_{c,n+1} := \max \{ \phi , \mathcal{F}_{e,c} [a_{c,n}] \},
\end{array}
\Ff
where $\mathcal{F}_{e,c}$ was defined in \eqref{eq:def_mapping}. The maximum is to be taken at each point $(y,z)$. 
\begin{lem}\label{lem:a1}
	The sequence $(a_{c,n})_{n\in\N}$ defined by~\eqref{eq:def_acn} is nondecreasing and satisfies $0<a_{c,n}< \bar p$ for $n\geq1$. 
Moreover, $a_{c,n}(y,z)$ is periodic in $y$, nonincreasing with respect to $c$ and $z$ and satisfies 
$a_{c,n}(y,+\infty)\equiv0$ uniformly with respect to $y$.
Lastly, $a_{c,n} (y,z+y\cdot e)$ is uniformly continuous 
in $y\in\R^N$, uniformly with respect to $z\in\R$, $n \in \mathbb{N}$ and $c\in\R$.
\end{lem}
\begin{proof}
Firstly, recall from Proposition~\ref{prop:Fec}$(ii)$ that the operator 
$\F{c}$ is order-preserving.
By recursion, one readily checks that the sequence $(a_{c,n})_{n\in\N}$ 
is nondecreasing.
Moreover, $0 <a_{c,n}<\bar p$ for $n\geq1$, always by Proposition~\ref{prop:Fec}$(ii)$. 
Another consequence of \eqref{eq:def_mapping1} and the comparison principle is that if $V(y,z)$ is monotone in $z$ then so is 
$\F{c}[V](y,z)$; whence the monotonicity of $a_{c,n} (y,z)$ with respect to $z$.

Let us now investigate the monotonicity with respect to $c$. We derive it
by noting that if $c_1 < c_2$, then \eqref{eq:def_mapping1} yields
\Fi{c1c21}
\F{c_1}[V](y,z+y\cdot e-c_1)=\F{c_2}[V](y,z+y\cdot e-c_2),
\Ff
for any function $V$. If furthermore $V (y,z)$ is nonincreasing in its second variable, then so is~$\F{c_2} [V]$ and we deduce that
$$\mathcal{F}_{e,c_1} [V] \geq \mathcal{F}_{e,c_2} [V].$$
Thus, owing to the monotonicity of the $\F{c}$,
the monotonicity of $a_{c,n}$ with respect to~$c$ follows by iteration.

Next, we want to show that $a_{c,n} (y,+\infty) =0$. 
This is an easy consequence of the same property for $\phi$,
but we now derive a quantitative estimate which will prove useful in the sequel.
For this, we observe that, for any fixed $\lambda>0$,
there exists a supersolution 
of \eqref{eq:parabolic} of the type $e^{-\lambda ( x \cdot e - \overline{c} t)}$,
provided $\overline c$ is sufficiently large.
Namely, by bounding~$f(x,u)$ by a linear function $Ku$ and also using  the boundedness of the components of the diffusion matrix and their derivatives, 
we can find $\overline c$ 
such that $e^{-\lambda ( x \cdot e - \overline{c} t)}$ satisfies
$$\partial_t u\geq  \text{div} (A (x) \nabla u)  + Ku,\quad
t \in \mathbb{R},\ x \in \mathbb{R}^N.$$
Let us show that if $V$
and $C>0$ satisfy
$$\forall (y,z)\in\R^{N+1},\quad
V (y, z) \leq C e^{-\lambda z},$$
then there holds
$$\forall (y,z)\in \R^{N+1},\quad
\mathcal{F}_{e,c}[V](y,z) \leq C  e^{\lambda ( \overline{c}-c) }  e^{ - \lambda  z}.$$
Indeed, we have that
$$V ( x , z+ x \cdot e - y \cdot e + c) \leq \big(C  e^{ - \lambda (z - y \cdot e + c)}\big)
e^{ - \lambda x \cdot e},$$
whence
$$\mathcal{F}_{e,c}[V](y,z)=v(1,y;V (x , z+ x \cdot e  - y \cdot e + c))
\leq C e^{-\lambda (z +c  -\overline{c})}.$$
Up to increasing $\overline{c}$, we can assume without loss of generality that $\overline{c} \geq c$.
Now, for any~$C \geq\max\bar p$, we have that 
$$\forall (y,z)\in\R^{N+1},\quad
\phi (y, z) \leq C e^{-\lambda z}.$$
As a consequence
$$\forall (y,z)\in \R^{N+1},\quad
a_{c,1}(y,z) = \max \{ \phi , \mathcal{F}_{e,c} [\phi] \}
 \leq C  e^{\lambda ( \overline{c}-c) }  e^{ - \lambda  z},$$
and therefore, by iteration,
\Fi{acn<1}
\forall n \in \N, \ \forall (y,z)\in \R^{N+1},\quad a_{c,n}(y,z)
 \leq C e^{n \lambda (\overline{c}-c) } e^{- \lambda z}.
\Ff
In particular $a_{c,n}(y,+\infty)=0$   uniformly with respect to $y$; however, this limit may not be uniform with respect to $c$ nor to $n$.

Finally, we point out that the uniform continuity in the crossed variables follows from our choice of $\phi$ and parabolic estimates. Indeed, the function 
$$y \mapsto  \mathcal{F}_{e,c} [a_{c,n-1}] (y, z + y \cdot e) = v (1,y; a_{c,n-1} (x,z + x \cdot e + c))$$
is not only uniformly continuous but also $C^2$, 
and its derivatives are uniformly bounded by some constant which only depends on the terms in the equation \eqref{eq:parabolic}
as well as $\max\bar p$. 
Recalling that $a_{c,n}$ is the maximum of $\mathcal{F}_{e,c} [a_{c,n-1}]$ and $\phi$, the latter being also uniformly continuous, we reach the desired conclusion.
\end{proof}
 From Lemma~\ref{lem:a1} and the fact that the mapping $\mathcal{F}_{e,c}$ preserves spatial periodicity, one readily infers the following.
\begin{lem}\label{lem:a2}
The  pointwise limit
 $$a_c(y,z):=\lim_{n\to+\infty} a_{c,n}(y,z),$$
 is well-defined,
 fulfils $\phi \leq a_c \leq \bar p$ and $a_c(y,z)$ is periodic in $y$ and nonincreasing 
with respect to both $z$ and $c$. 

Moreover, the convergence
 $$a_{c,n}(y,z+y\cdot e)\to a_c(y,z+y\cdot e)
 \quad\text{as $n \to +\infty$}$$
holds locally uniformly in $y\in\R^N$, but still pointwise in $z\in \mathbb{R}^N$.
\end{lem}

We emphasize that no regularity properties could be expected for 
$a_c$ with respect to the second variable.
Let us further note that, as a byproduct of the proof of Lemma~\ref{lem:a1}, and more specifically of \eqref{c1c21}, we deduce by iteration that
	\begin{equation}\label{acc'1}
	\forall c<c',\ n\in\N,\quad
	a_{c,n}(\cdot,\cdot + n(c'-c))\leq a_{c',n} .
	\end{equation}
This will be used in later arguments, in particular in the proof of Lemma~\ref{lem:limit1} below.	

\subsection{Defining $c^*$}\label{sec:c^*}

We want to define $c^*$ as the largest $c$ such that $a_c \equiv \bar p$, where $a_c$ comes from Lemma~\ref{lem:a2}. This is the purpose of the following lemma.
\begin{lem}\label{lem:ac1}
For any $c \in \mathbb{R}$, the function $a_c$ satisfies $a_c (y,-\infty) =  \bar p (y)$
uniformly with respect to $y\in[0,1]^N$. 
Moreover,
\begin{enumerate}[$(i)$]
\item $a_c \equiv \bar p$ for $-c$ large enough;
\item $a_c \not\equiv \bar p$ for $c$ large enough.
\end{enumerate}
In particular, the following is a well-defined real number:
$$c^* := \sup \{ c \in \mathbb{R} \ : \ a_c (y,z) \equiv \bar p(y) \}.$$
\end{lem}
\begin{proof}
We first prove that, for $-c$ large enough,
\Fi{Fn11}
(\mathcal{F}_{e,c})^n [\phi] (y,z) \to \bar p(y)\quad\text{as $n \to +\infty$},
\Ff
uniformly with respect to $y\in[0,1]^N$ and $z \in (-\infty,Z_0]$, for any $Z_0 \in \mathbb{R}$. 
In particular, because $a_{c,n} \geq (\mathcal{F}_{e,c})^n [\phi]$ by the monotonicity of $\F{c}$, 
this will yield statement $(i)$ of the lemma.

In order to show \eqref{Fn11}, we first introduce, in a similar fashion as in the proof of 
Lemma~\ref{lem:a1}, two real numbers $\lambda >0$ and $\overline{c}$ large enough such that the function $e^{\lambda(x \cdot e + \overline{c} t)}$ satisfies the parabolic inequality
$$\partial_t u \geq \text{div} (A \nabla u ) + K u.$$
Here $K$ is the supremum with respect to $x$
of the Lipschitz constants of $u\mapsto f(x,u)$.

Next, we let $\psi (t,x)$ be the solution of \eqref{eq:parabolic}
emerging from the initial datum $\phi(x,-\infty)-\delta$, where~$\delta$ is the positive
constant in condition~\eqref{ass:phi}, that is, such that
$\phi(x,-\infty)-\delta$ lies in the basin of attraction of $\bar p$. Hence
$\psi (t, \cdot) \to \bar p$ uniformly as $t \to +\infty$.
The choice of $\lambda$ and $\overline{c}$ imply that, for any $\gamma>0$, the function
$$u_\gamma (t,x) := \psi (t,x) - \gamma e^{\lambda (x \cdot e  + \overline{c} t)}$$
is a subsolution of~\eqref{eq:parabolic}. Let us now pick $C$ large enough such that
$$\forall(y,z)\in\R^{N+1},\quad
\phi (y,z) \geq \phi(y,-\infty)-\delta - C e^{\lambda z},$$
and thus, for any given $c\in\R$,
$$\phi (x , z+x\cdot e) \geq 
\phi(x,-\infty)-\delta - C e^{\lambda (z+x \cdot e)}=u_{Ce^{\lambda z}}(0,x).$$
Now, iterating \eqref{eq:def_mapping1} one gets
$$
\forall n\in\N,\quad
(\mathcal{F}_{e,c})^n[V](y,z+y\cdot e-nc)=v(n,y;V (x , z + x \cdot e)).
$$
It then follows from the comparison principle that
$(\mathcal{F}_{e,c})^n [\phi ] (y,z-nc) \geq u_{Ce^{\lambda z}}(n,y)$, that~is, 
\Fi{Fn>psi}
(\mathcal{F}_{e,c})^n [\phi ] (y,z) \geq 
\psi (n,y) -  C e^{\lambda[z+y\cdot e+n (c+ \overline{c})]}.
\Ff
From one hand, this inequality implies that if 
$c < - \overline{c}$ then \eqref{Fn11} holds uniformly with respect to $y\in[0,1]^N$ and 
$z \in (-\infty,Z_0]$, for any $Z_0 \in \R$, whence statement $(i)$ of the lemma. From the other hand, if $c\geq - \overline{c}$ we derive
$$a_{c,n} (y,-2n(c+ \overline{c}+1)) \geq \psi (n,y) -  
C e^{-n \lambda (c+ \overline{c}+2)+\lambda y\cdot e}
\geq \psi (n,y) -  C e^{\lambda(-2n+y\cdot e)}.$$
Because the sequence $(a_{c,n})_{n\in\N}$ is nondecreasing and converges to $a_c$, 
we get that
$$a_{c} (y,-2n(c+ \overline{c}+1)) \geq \psi (n,y) -  C e^{\lambda(-2n+y\cdot e)} ,$$
for any $n \in \mathbb{N}$. 
Passing to the limit as $n \to +\infty$ and
recalling that $a_c$ is monotone with respect to its second variable,
we infer that $a_c(y,-\infty)=\bar p(y)$ uniformly with respect to
$y\in[0,1]^N$.

It remains to prove statement $(ii)$. Fix $\lambda>0$.
Because $\phi$ satisfies \eqref{ass:phi},
for~$C:=\max\bar p$ there holds that
$\phi (y, z) \leq C e^{-\lambda z}$ for all $(y,z)\in\R^{N+1}$.
As seen in the proof of Lemma~\ref{lem:a1}, this implies that
\eqref{acn<1} holds for all $c$ smaller than or equal
a suitable value~$\overline{c}$, and then
in particular for $c = \overline{c}$, i.e., 
$a_{\overline{c},n}(y,z)\leq Ce^{-\lambda z}$ for all $n\in\N$.
As a consequence, $a_{\overline{c}} \not\equiv \bar p$ and, by monotonicity with respect to~$c$,
we also have that $a_c \not \equiv \bar p$ if $c\geq\bar c$.
\end{proof}

We see now that, while $c^*$ is the supremum of the speeds $c$ such that 
$a_c \equiv \bar p$, it actually holds that $a_{c^*} \not \equiv \bar p$. 
This will be crucial for the construction of the~front.
\begin{lem}\label{lem:ac2}
The following properties are equivalent:
\begin{enumerate}[$(i)$]
\item $  c < c^*$,
\item $  a_c \equiv \bar p$, 
\item $ \exists n_0 \in \mathbb{N},  \ \exists z_0 >0, \  
\forall y\in[0,1]^N,\quad
  a_{c,n_0} (y,z_0) > \phi (y,-\infty).$
\end{enumerate}
In particular, in the case $c = c^*$, we have that 
for all $n\in\N$ and $z > 0$, there exists $y\in[0,1]^N$
such that $a_{c^*,n} (y,z)\leq\phi (y,-\infty)$.
\end{lem}
\begin{proof}
By definition of $c^*$ and monotonicity of $a_c$ with respect to $c$,
we already know that~$(i)$ implies $(ii)$. 
We also immediately see that $(ii)$ implies $(iii)$, using the fact that
$a_{c,n}(y,z)$ is nonincreasing in $z$ and
$a_{c,n}(y,z+y\cdot e)\to a_c(y,z+y\cdot e)$ as $n\to+\infty$
uniformly with respect to $y\in[0,1]^N$ (see Lemma~\ref{lem:a2}).

It remains to prove that $(iii)$ implies $(i)$. We assume that $(iii)$ holds 
and we start by showing~$(ii)$, 
which will serve as an intermediate step. 
Thanks to the monotonicity with respect to~$z$ and the fact that 
$a_{c,n_0}>0$ and $\phi (\cdot ,z) =0$ for $z \geq 0$, 
we get
$$\forall n \geq n_0, \ \forall (y,z) \in \mathbb{R}^{N+1},\quad
a_{c,n} (y, z + z_0) > \phi (y,z).$$
Since the operator $\mathcal{F}_{e,c}$ is order preserving, we also get that
$$\forall (y,z) \in \mathbb{R}^{N+1},\quad
a_{c,n_0 +1} (y,z+ z_0)\geq\mathcal{F}_{e,c} [a_{c,n_0}](y,z+ z_0)
 \geq \mathcal{F}_{e,c} [\phi] (y,z) .$$
It follows from the two inequalities above that
$$\forall (y,z) \in \mathbb{R}^{N+1},\quad
a_{c,n_0 +1} (y,z+ z_0) \geq a_{c,1} (y,z).$$
A straightforward induction leads to
$$\forall m \geq 0,\
\forall (y,z) \in \mathbb{R}^{N+1},\quad
a_{c,n_0 +m} (y,z+z_0) \geq a_{c,m} (y,z).$$
Passing to the limit $m \to +\infty$ on both sides, we infer that
$$\forall (y,z) \in \mathbb{R}^{N+1},\quad
a_{c} (y,z+z_0) \geq a_c (y,z).$$
Recalling that $z_0 >0$ and that $a_c$ is nonincreasing with respect to $z$, 
we find that $a_c (y,z) = a_c (y)$ does not depend on $z$. Since we know by Lemma~\ref{lem:ac1} that $a_c (\cdot,-\infty) \equiv \bar p(\cdot)$, we conclude that $a_c \equiv \bar p$. We have shown that $(iii)$ implies $(ii)$.

Next we show that the set of values of $c$ such that $(iii)$ holds is open. Using \eqref{eq:def_mapping1}, 
it is readily seen by iteration that, for any fixed $n \in \mathbb{N}$, the function $a_{c,n}$ inherits from~$\phi$ the continuity with respect to the variable $(y,z)$ (though 
this is not uniform with respect to $n\in\N$).
From this, by another iterative argument and \eqref{c1c21}, one deduces that
$a_{c,n}(y,z)\to a_{c_0,n}(y,z)$ locally uniformly in $(y,z)$ as $c\to c_0$,
for every $n\in\N$. Openness follows.

We are now in the position to conclude the proof of Lemma~\ref{lem:ac2}. 
Assume that $(iii)$ holds for some $c$. From what we have just proved, we know that $(iii)$ holds true for some $c' >c$, and thus $(ii)$ holds for $c'$ too. By the definition of $c^*$, we have that $c^* \geq c' > c$, that is, $(i$) holds.
\end{proof}

Before proceeding we have to check that $c^*$ is intrinsic to \eqref{eq:parabolic} and does not 
depend on~$\phi$. This will be useful later on, when going back to the continuous case and 
more specifically to check that the speed 
of the discrete front we shall obtain does not depend on the choice of the time step
of the discretization.
\begin{lem}\label{lem:indep1}
The speed $c^*$ does not depend on the choice of $\phi$ satisfying the 
properties~\eqref{ass:phi}.
\end{lem}
\begin{proof}
Consider two admissible functions $\phi$ and $\hat{\phi}$ for the conditions~\eqref{ass:phi}.
Let $a_{c,n}$, $\hat{a}_{c,n}$ and $c^*$, $\hat{c}^*$ 
denote the functions and constants
constructed as above, starting from~$\phi$, $\hat{\phi}$ respectively.
Take an arbitrary $c\in\R$.
Using the first part of
Lemma~\ref{lem:ac1} and the fact that $a_{c,n} (y,z + y\cdot e) \to a_c (y , z + y \cdot e)$ 
locally uniformly in $y$ as $n \to +\infty$, 
we can find $\bar z <0$ and $\bar n\in\N$ such that
$$\inf_{y\in[0,1]^N}\left(a_{c,\bar n}(y,\bar z+y\cdot e) - \hat\phi(y,-\infty)\right)>0.$$
Because $|y\cdot e|\leq \sqrt{N}$ if $y\in[0,1]^N$, 
one readily deduces that $a_{c,\bar n} (y,z-\sqrt{N} + \bar z) > \hat \phi (y,z)$ for all $(y,z) \in \R^{N+1}$, whence $a_{c,n} (\cdot, \cdot - \sqrt{N} + \bar z ) > \hat \phi$ for all $n \geq \bar n$ by the monotonicity in~$n$.
It follows that
$$a_{c,\bar n+1}(\cdot,\cdot - \sqrt{N} +\bar z)\geq \max\{\hat\phi,\F{c}[a_{c,\bar n}](\cdot,\cdot - \sqrt{N} +\bar z)\}
\geq \max\{\hat\phi,\F{c}[\hat\phi]\}={\hat a}_{c,1}.$$
By iteration we eventually infer that $a_{c,\bar n+m}(\cdot,\cdot - \sqrt{N}+\bar z)\geq\hat a_{c,m}$
for all $m\in\N$. This implies that $c^*\geq \hat c^*$.
Switching the roles of $\phi$ and $\hat\phi$ we get the reverse inequality.
\end{proof}

\section{A discrete travelling front with speed $c^*$}
\label{sec:discreteTF}

Under the Assumption~\ref{ass:mix}, 
we have constructed in the previous section a candidate speed~$c^*$ for the existence
of a pulsating travelling front. 
In the current one we show that there exists a discrete travelling front in the direction $e$ with speed~$c^*$ connecting~$\bar p$ to some stable periodic steady state
(in the sense of Definition~\ref{def:discrete_front}). 
To derive the stability of the latter we will make use of
the additional Assumption~\ref{ass:speeds}.
We recall that in order to define the minimal 
speeds~$\overline{c}_{q}$ and~$\underline{c}_{q}$ appearing in
Assumption~\ref{ass:speeds},
we have shown after the statement of Theorem~\ref{thm:mono} that the hypothesis
there is guaranteed by Assumption~\ref{ass:multi}. However, 
this was achieved without using the linear stability hypothesis in Assumption~\ref{ass:multi}
and therefore~$\overline{c}_{q}$ and~$\underline{c}_{q}$ are well 
defined under Assumption~\ref{ass:mix}~too.

The strategy is as follows. For $c<c^*$, Lemma~\ref{lem:ac2} implies
that $a_{c,n} > \phi$ for $ n $ sufficiently large. We deduce that the nondecreasing sequence 
$(a_{c,n})_{n\in\N }$ is eventually given by the recursion 
$a_{c,n} = \mathcal{F}_{e,c} [a_{c,n-1}]$.
Roughly speaking, we have constructed a solution of \eqref{eq:parabolic} 
which is non-decreasing with respect to $1$-time steps in the frame 
moving with speed $c$ in the direction $e$. We now want to pass to the limit as $c \nearrow c^*$ in order to get a fixed point for $\mathcal{F}_{e,c^*}$ and, ultimately, a pulsating travelling front in the direction~$e$. To achieve this, we shall need to capture such solutions
at a suitable time step, and suitably translated.
\begin{rmk}
The equivalent argument in the continuous case of what we 
are doing here is to construct a family of functions $U_c$ such that $U_c (x, x\cdot e - ct)$ is a subsolution of \eqref{eq:parabolic}, and to use this family and a limit argument to find a pulsating front. Notice that an inherent difficulty in such an argument is that a subsolution does not satisfy regularity estimates in general. We face a similar difficulty in the discrete framework.
\end{rmk}

\subsection{Choosing a diagonal sequence as $c \nearrow  c^*$}\label{sec:subsequence}

Consider the function $\phi$ satisfying~\eqref{ass:phi} 
from which we initialize the construction of the sequence~$(a_{c,n})_{n\in\N}$.
	
The first step in order to pass to the limit as $c \nearrow c^*$ is to capture the sequence at a suitable iteration, and roughly at the point where it `crosses' the limit~$\phi(\cdot,-\infty)$, which, we recall, lies in the basin of attraction
of $\bar p$.
\begin{lem}\label{lem:n(c)}
For $c < c^*$, there exists $n(c) \in \N$ such that, for all $n \geq n(c)$, the quantity
$$z_{c,n}:=\sup \{z\ :\ a_{c,n} (y,z+y\cdot e)> \phi(y,-\infty)\ \text{ for all }y\in[0,1]^N\}$$
is a  well-defined real number. In addition, there holds
\Fi{eq:nc1}
\forall m\geq 0,\quad
a_{c,n(c)+m}=(\mathcal{F}_{e,c})^m [a_{c,n(c)}]\geq a_{c,n(c)+m-1},
\Ff
\Fi{eq:zcn1}
\forall 0\leq  m\leq 1/\sqrt{c^* - c},\quad
0 \leq z_{c,n(c) + m} - z_{c,n(c)} \leq 2 \sqrt{c^* - c}.
\Ff
\end{lem}

While property \eqref{eq:nc1} holds for any $c<c^*$ provided $n(c)$ is sufficiently large, 
the same is not true for \eqref{eq:zcn1}.
The latter will play a crucial role for getting a travelling front in the limit.
Loosely speaking, it guarantees that, as $c\nearrow c^*$, there exists an 
index~$n(c)$ starting from which the ``crossing point'' 
$z_{c,n}$ moves very little along an arbitrary large number of iterations.
\begin{proof}[Proof of Lemma \ref{lem:n(c)}]
Fix $c <c^*$. 
First of all, from the equivalence between~$(i)$ and~$(iii)$ in Lemma~\ref{lem:ac2}, we know that there exists $n(c)$ such that  $a_{c,n} > \phi$
for $ n \geq  n(c)$. We deduce that the nondecreasing sequence 
$(a_{c,n})_{n \geq  n(c)}$ is simply given by the recursion $a_{c,n} = \mathcal{F}_{e,c} [a_{c,n-1}]$, that is
property \eqref{eq:nc1}.
Now, Lemma~\ref{lem:a1} implies that the set 
$$\{z\ :\ a_{c,n} (y,z+y\cdot e)> \phi(y,-\infty)\
\text{ for all }y\in[0,1]^N\}$$ is either a left half-line or the empty set, while Lemmas~\ref{lem:a2}-\ref{lem:ac1}
show that it is nonempty for $n$ sufficiently large.
As a consequence,
up to increasing $ n(c)$ if need be, 
its supremum $z_{c,n}$ is well-defined and finite for~$n \geq n(c)$. 

It remains to prove \eqref{eq:zcn1}, for which we can assume that $c\geq c^*-1$.
We claim that
\Fi{eq:z/n}
\limsup_{n \to +\infty} \frac{z_{c,n}}{n} \leq c^* - c.
\Ff
Indeed by the definition of $z_{c,n}$ and \eqref{acc'1}, for $n\geq n(c)$ we get
\[\begin{split}
0 &< \min_{y\in[0,1]^N} \big(a_{c,n} (y,z_{c,n}-1+ y \cdot e) - \phi(y,-\infty)\big)\\
&\leq\min_{y\in[0,1]^N} \big(a_{c^*,n} (y,z_{c,n}+n(c-c^*) + y \cdot e -1 ) - \phi(y,-\infty)\big).
\end{split}\]
Hence if \eqref{eq:z/n} does not hold,
we would find a large $n$ contradicting the last statement of Lemma \ref{lem:ac2}.

Next, let $N(c)\geq1$ be the integer part of $1/\sqrt{c^* - c}$. 
Owing to \eqref{eq:z/n}, we can further increase $n(c)$ to ensure that
$$z_{c,n(c) + N(c)} - z_{c,n(c)} \leq 2 N(c) (c^* - c).$$
Moreover, we know that $z_{c,n+1} \geq z_{c,n} $ for all $c$ and $n$, due to the monotonicity of $a_{c,n}$ with respect to $n$. 
In particular, for any integer $0 \leq m \leq N(c)$, we also have that
$$0 \leq z_{c,n(c) + m} - z_{c,n(c)} \leq 2 N(c) (c^* - c),$$
from which we deduce \eqref{eq:zcn1}.
\end{proof}
In the next lemma, we state what we obtain when passing to the limit as $c \nearrow c^*$.
\begin{lem}\label{lem:limit1}
There exists a lower semicontinuous
function $a^*(y,z)$ satisfying the following properties:
\begin{enumerate}[$(i)$]   
   \item $a^* (y,z+y\cdot e)$ is uniformly continuous 
   in $y\in\R^N$, uniformly with respect to $z\in\R$;
   \item $a^*(y,z)$ is periodic in $y$ and nonincreasing in $z$;
   \item $(\mathcal{F}_{e,c^*})^n [a^*]$ is nondecreasing with respect to $n$;
   \item $\lim_{n \to+\infty}
   \Big( \max_{y \in [0,1]^N} \big(\bar p(y) -(\mathcal{F}_{e,c^*})^n [a^*] (y,y \cdot e)\big)\Big) > 0$;
   \item $(\mathcal{F}_{e,c^*})^n [a^*](\cdot,-\infty)\nearrow \bar p$ uniformly
   as $n\to+\infty$;
  \item $(\mathcal{F}_{e,c^*})^n [a^*] (\cdot, +\infty) \nearrow p$ uniformly
  as $n \to +\infty$, where 
  $0 \leq p < \bar p$ is a periodic steady state of~\eqref{eq:parabolic}.
\end{enumerate}
\end{lem}
Thanks to our previous results, 
we know that the properties $(i)$-$(iii)$
 are fulfilled with $c^*$ and $a^*$ replaced respectively by any $c < c^*$ and $a_{c,n}$
 with $n$ sufficiently large. 
 In order to get $(iv)$-$(vi)$ we need to pass to the limit $c\nearrow c^*$ by picking 
 the $a_{c,n}$ at a suitable iteration~$n$. 
 The choice will be $n=n(c)$ given by Lemma \ref{lem:n(c)},
 which fulfils the key property \eqref{eq:zcn1}. 
When passing to the limit, we shall face the problem of the lack of regularity in the $z$-variable. 
This will be handled by considering the following relaxed notion of limit.

\begin{lem}\label{lem:diagonal}
	Let $(\alpha_n)_{n\in\N}$ be a bounded sequence of functions from $\R^N\times\R$ to $\R$ such that 
	$\alpha_n(y,z)$  is periodic in $y$ and nonincreasing in $z$, 
	and $\alpha_n(y,z+y\cdot e)$ is uniformly 
	continuous in $y\in\R^N$, uniformly with respect to $z$ and $n$. 
	Then there exists a subsequence $(\alpha_{n_k})_{k\in\N}$ such that the following double limit
	exists locally uniformly in~$y\in\R^N$:
	$$\beta(y,z):=\lim_{ \Q\ni \zeta \to z^+}
	\Big(\lim_{k\to+\infty}\alpha_{n_k}(y,\zeta+y\cdot e)\Big).$$
	
	Furthermore, $\beta(y,z)$ is uniformly continuous 
	in $y\in\R^N$ uniformly with respect to $z\in\R$.
	Finally, the function
	$\alpha^*(y,z):=\beta(y,z-y\cdot e)$ is periodic in $y$ and
	nonincreasing and lower semicontinuous in $z$.
\end{lem}

\begin{proof}
	Using a diagonal method, we can find a subsequence 
	$\alpha_{n_k}(y,\zeta+y\cdot e)$ converging locally uniformly in $y\in\R$ to some function $\tilde\beta(y,\zeta)$ 
	for all $\zeta\in\Q$. The function $\tilde\beta(y,\zeta)$ is uniformly continuous in $y$
	uniformly with respect to $\zeta\in\Z$.
	We then define $\beta:\R^N\times\R\to\R$ by setting
	$$\beta(y, z) := \lim_{ \Q\ni \zeta \to z^+} \tilde\beta (y, \zeta).$$
	This limit exists thanks to the monotonicity with respect to $z$, and it is locally uniform with respect to $y$
	by equicontinuity. We point out 
	that $\beta\leq\tilde \beta$ on $\R^N\times\Q$, 
	but equality may fail. 	
	We also see that $\beta(y,z)$ is uniformly continuous 
	in $y\in\R^N$ uniformly with respect to $z\in\R$, and it is nonincreasing and 
	lower semicontinuous in $z$. 
	
	Next, we define $\alpha^*(y,z):=\beta(y,z-y\cdot e)$. 
	We need to show that $\alpha^*(y,z)$ is periodic in~$y$.
	Fix $(y,z)\in\R^N\times\R$ and $L \in \mathbb{Z}^N$.
	Then, using the periodicity of $\alpha_n$,
	for every $\zeta,\zeta'\in\Q$ satisfying 
	$\zeta< z$ and $\zeta'+L \cdot e > \zeta$, 
	we get
	$$\alpha_n(y, \zeta+ y \cdot e) \geq \alpha_n (y+L ,\zeta'+ (y+L) \cdot e).$$
	Passing to the limit along the subsequence $\alpha_{n_k}$ we deduce
	$$\tilde\beta (y, \zeta) \geq 
	\tilde\beta (y + L , \zeta').$$
	Now we let $\Q\ni\zeta \to z^+$ and $\Q\ni\zeta' \to( z - L \cdot e)^+$ and we derive
	$$\beta (y, z) \geq 
	\beta (y + L , z - L \cdot e).$$
	That is, $\alpha^*(y, z+y \cdot e) \geq \alpha^* (y+L,z+y \cdot e)$. 
	Because $y$ and $z$ are arbitrary, this means that $\alpha^*\geq \alpha^* (\cdot+L,\cdot)$
	for all $L\in\Z^N$, i.e., $\alpha^*$ is periodic in its first variable.
\end{proof}

\begin{proof}[Proof of Lemma~\ref{lem:limit1}]
Consider the family of functions $(a_{c,n(c)}(y,z+z_{c,n(c)}))_{c<c^*}$, with $n(c)$, $z_{c,n(c)}$ given by 
	Lemma \ref{lem:n(c)}. From Lemma~\ref{lem:a1}, we know that this family is uniformly bounded by 0 and $\max\bar p$, and that any element $a_{c,n(c)}$ is periodic in the first variable and nonincreasing in the second one. 
	Moreover, the functions 
$a_{c,n(c)}(y,z+y\cdot e)$ are uniformly continuous in $y\in\R^N$,
uniformly with respect to $z\in\R$ and $c \in \mathbb{R}$, due to Lemma~\ref{lem:a1}.
In particular, any sequence extracted from this family fulfils the hypotheses of Lemma~\ref{lem:diagonal}.
	Then, there exists a  sequence $c^k\nearrow c^*$ such that
	the following limits exist locally uniformly in $y\in\R^N$:
	\Fi{def:a*}
	a^*(y,z+y\cdot e):=\lim_{ \Q\ni \zeta \to z^+}\Big(\lim_{k\to+\infty}
	a_{c^k,n(c^k)}(y,\zeta+z_{c^k,n(c^k)}+y\cdot e)\Big).
	\Ff
	We further know that the function $a^*$ satisfies the desired properties $(i)$-$(ii)$.
	The definition of $z_{c,n(c)}$ translates into the following normalization conditions:
	\Fi{atilde>}
	\forall z<0,\quad
        \min_{y\in[0,1]^N} \big(a^*(y,z+y\cdot e) - \phi(y,-\infty)\big) \geq 0,
        \Ff
	\Fi{atilde<}
        \min_{y\in[0,1]^N}\big(a^*(y,y\cdot e) -\phi(y,-\infty)\big) \leq 0,
        \Ff
	where we have used the monotonicity in $z$ and for the second one also
	the locally uniform convergence with respect to~$y$. 
		
	Let us check property $(iii)$. Using the continuity property of
	Proposition \ref{prop:Fec} together with \eqref{c1c21} we obtain 
	\[\begin{split}
	\mathcal{F}_{e,c^*} [a^*] (y, z + y \cdot e - c^*) &=
	\lim_{ \Q\ni \zeta \to (z-c^*)^+}\Big(\lim_{k\to+\infty}
	\mathcal{F}_{e,c^*} [a_{c^k,n(c^k)}](y,\zeta+z_{c^k,n(c^k)}+y\cdot e)\Big)\\
	&=\lim_{ \Q\ni \zeta \to (z-c^*)^+}\Big(\lim_{k\to+\infty}
	\mathcal{F}_{e,c^k} [a_{c^k,n(c^k)}](y,\zeta+z_{c^k,n(c^k)}+y\cdot e+c^*-c^k)\Big).
	\end{split}\]
	We now use property \eqref{eq:nc1} to deduce that the latter term is larger
	than or equal to
	\[
	\limsup_{ \Q\ni \zeta \to (z-c^*)^+}\Big(\limsup_{k\to+\infty}
	 a_{c^k,n(c^k)}(y,\zeta+z_{c^k,n(c^k)}+y\cdot e+c^*-c^k)\Big),\]
	which, in turn, is larger than or equal to
	\[\lim_{k\to+\infty}
	a_{c^k,n(c^k)}(y,\zeta'+z_{c^k,n(c^k)}+y\cdot e),
	\]
	for any rational $\zeta'>z-c^*$. 
	Letting $\Q\ni \zeta' \to (z-c^*)^+$, we eventually conclude that
	$$\mathcal{F}_{e,c^*} [a^*] (y, z + y \cdot e - c^*)\geq 
	a^*(y, z + y \cdot e - c^*).$$
	Property $(iii)$ then follows by iteration.
	
Next, fix $m\in\N$ and a positive $\zeta\in\Q$. 
We know by \eqref{eq:nc1} that, for every $k\in\N$ and~$y\in\R^N$,
\[\begin{split}
a_{c^k, n(c^k)+m} (y,\zeta+z_{c^k, n(c^k)}+y\cdot e ) &=
(\mathcal{F}_{e,c^k})^m [a_{c^k, n(c^k)}] (y,\zeta+z_{c^k, n(c^k)}+y\cdot e )\\
&\geq(\mathcal{F}_{e,c^*})^m [a_{c^k, n(c^k)}] (y,\zeta+z_{c^k, n(c^k)}+y\cdot e ).\\
\end{split}\]
Let $k$ large enough so that $1/\sqrt{c^* -c^k }\geq m$ and 
$2 \sqrt{c^* -c^k }<\zeta$. We deduce from 
\eqref{eq:zcn1} that $\zeta+z_{c^k, n(c^k)}>z_{c^k,n(c^k) + m}$ and thus
$$\min_{y\in[0,1]^N}
\big((\mathcal{F}_{e,c^*})^m [a_{c^k, n(c^k)}] (y,\zeta+z_{c^k, n(c^k)}+y\cdot e ) - \phi(y,-\infty)\big)
\leq 0.$$
Letting now $k\to+\infty$ and next $\zeta\to0^+$ and using the continuity of $\mathcal{F}_{e,c}$ (hence of~$(\mathcal{F}_{e,c})^m$) in the locally uniform topology, we eventually obtain
$$\forall m\in\N,\quad
\min_{y\in[0,1]^N}
\big((\mathcal{F}_{e,c^*})^m [a^*] (y,y\cdot e ) - \phi(y,-\infty)\big) \leq 0,$$
from which property $(iv)$ readily follows.

It remains to look into the asymptotics of $a^*$ as $z \to \pm \infty$. 
We define the left limit
$$a_{\ell} (y) := \lim_{ z \to -\infty} a^* (y,z + y \cdot e),$$
which exists by the monotonicity of $a^*(y,z)$ with respect to $z$, and it is 
locally uniform in~$y$. 
Then, by monotonicity and periodicity, we deduce that the limit $a^* (y,-\infty)=a_\ell(y)$ 
holds uniformly in $y$.
The function $a_\ell$ is continuous and periodic.
Moreover, the normalization condition \eqref{atilde>} yields $a_\ell(y)\geq\phi(y,-\infty)$. 
Finally, by the continuity of~$\mathcal{F}_{e,c^*}$ we get, for all $y \in [0,1]^N$, 
$$\mathcal{F}_{e,c^*}[a_{\ell}](y)=
\lim_{ z \to -\infty} \mathcal{F}_{e,c^*}[a^*] (y,z + y \cdot e)\geq a_{\ell}(y).$$
This means that the sequence $((\mathcal{F}_{e,c^*})^n[a_{\ell}])_{n\in\N}$ is 
nondecreasing. Because $a_{\ell}$ is independent of $z$,
by the definition \eqref{eq:def_mapping} we see that 
$(\mathcal{F}_{e,c^*})^n[a_{\ell}]$ reduces to $(\mathcal{F}_{e,0})^n[a_{\ell}]$, 
that is, to the solution of~\eqref{eq:parabolic} 
with initial datum $a_{\ell}$ computed at time $t=n$. 
Then, because $a_{\ell}\geq\phi(y,-\infty)$ and recalling that 
the latter lies in the basin of attraction of $\bar p$,
we infer that $(\mathcal{F}_{e,0})^n[a_{\ell}]\to \bar p$ as $n\to+\infty$, 
and the limit is uniform thanks to
Proposition~\ref{prop:Fec}$(iv)$.

In a similar fashion, we define the (locally uniform) right limit
$$a_{r} (y) := \lim_{ z \to +\infty} a^* (y,z + y \cdot e).$$
As before, we see that the limit $a^* (y,+\infty)=a_{r}(y)$ is uniform
in $y$, it is continuous, periodic
and the sequence $((\mathcal{F}_{e,0})^n[a_{r}])_{n\in\N}$
is nondecreasing. Therefore, $(\mathcal{F}_{e,0})^n[a_{r}](y)$ converges 
uniformly as $n\to+\infty$
to a fixed point $p(y)$ of $\mathcal{F}_{e,0}$.
This means that the solution $u$ of~\eqref{eq:parabolic} 
with initial datum $p$ is 1-periodic in time and periodic in space and
therefore, by Proposition~\ref{prop:energy1}, it is actually stationary. 
We conclude that $(vi)$ holds, completing the proof of the lemma.
\end{proof}

\subsection{The uppermost pulsating front}\label{sec:uppermost}

From now on, $a^*$ will denote the function provided by Lemma \ref{lem:limit1}
and more specifically defined by \eqref{def:a*} for a suitable sequence
$c^k\nearrow c^*$. Next we show that the discrete front
is given by the limit of the iterations
$(\mathcal{F}_{e,c^*})^n [a^*]$. We shall further show that its limit state as~$z\to+\infty$ is stable.

\begin{lem}\label{lem:U1}
There holds that
$$(\mathcal{F}_{e,c^*})^n [a^*](y,z+y\cdot e)\to U^* (y,z+y\cdot e)\quad\text{as }\;
n\to+\infty,$$
locally uniformly in $y$ and pointwise in $z$, where $U^*(x,z)$ is nonincreasing in $z$ and it 
is a discrete travelling front 
connecting~$\bar p$ to some stable periodic steady state~$p^*< \bar p$, in the sense of 
Definition~\ref{def:discrete_front}.
\end{lem}
\begin{proof}
Let us observe that, because $((\mathcal{F}_{e,c^*})^n [a^*])_{n\in\N}$ is a nondecreasing sequence, it is already clear that it converges pointwise to some function $U^*(y,z)$ which is periodic in~$y$ and nonincreasing in $z$. By writing 
$$(\mathcal{F}_{e,c^*})^{n+1} [a^*] (y,z+y\cdot e) = 
\mathcal{F}_{e,c^*}\circ(\mathcal{F}_{e,c^*})^n [a^*] (y,z+y\cdot e),$$ 
we deduce from Proposition \ref{prop:Fec}$(iv)$ that
$(\mathcal{F}_{e,c^*})^n [a^*] (y,z+y\cdot e)$ converges as $n \to +\infty$ locally uniformly in~$y$, 
for any $z\in\R$. In particular, we can pass to the limit $n \to +\infty$ in the above equation and conclude that $U^*$ is a fixed point for $\mathcal{F}_{e,c^*}$.

Let us now turn to the asymptotics as $z \to \pm \infty$. We know from  
Lemma~\ref{lem:limit1}$(v)$ that 
$(\mathcal{F}_{e,c^*})^n [a^*] (\cdot , -\infty) \to \bar p$ as $n \to +\infty$. 
We can easily invert these limits using 
the continuity of $(\mathcal{F}_{e,c^*})^n$ and the uniformity 
of the limit $a^*(\cdot,-\infty)$, together with the
monotonicity of $U^*$ in the second variable.
This yields $U^*(\cdot,-\infty) \equiv \bar p$.

Next, property~$(iv)$ of Lemma~\ref{lem:limit1} implies that
$$\max_{y\in[0,1]^N}\big(\bar p(y)-U^*(y,y\cdot e)\big) >0.$$
Writing $U^*(y,+\infty)=\lim_{z\to+\infty}U^*(y,z+y\cdot e)$,
we deduce that the limit $U^*(\cdot,+\infty)$ is uniform and therefore
$\mathcal{F}_{e,0}[U^*](\cdot,+\infty)\equiv
\mathcal{F}_{e,c^*}[U^*](\cdot,+\infty)\equiv U^*(\cdot,+\infty)$. 
We also deduce from the previous inequality
that $U^*(\cdot,+\infty)\not\equiv\bar p $.
As seen in Proposition~\ref{prop:energy1},
any solution of~\eqref{eq:parabolic}
that is periodic in both time and space is actually constant in time.
Thus, $U^*(\cdot ,+\infty)$ is a periodic steady state of~\eqref{eq:parabolic},
denoted by~$p^*$, 
that satisfies $0\leq p^*<\bar p$, where the second inequality is strict due to the 
elliptic strong maximum principle.

It remains to check that $p^*$ is stable. 
We shall do this using Assumption~\ref{ass:speeds}.
Proceed by contradiction and assume that~$p^*$ is unstable. 
As seen after the statement of Theorem~\ref{thm:mono},
Assumption~\ref{ass:mix} guarantees the existence of a minimal (resp.~maximal)
stable periodic  steady state above (resp.~below)~$p^*$, denoted by
$p_+$ (resp.~$p_-$), and also that~\eqref{eq:parabolic} is of the 
monostable type between $p_-$ and $p^*$, as well as between $p^*$ and~$p_+$. 
As a consequence, Theorem~\ref{thm:mono}
provides two minimal speeds of fronts~$\overline{c}_{p^*}$ and~$\underline{c}_{p^*}$
connecting $p_+$ to $p^*$ and $p^*$ to $p_-$ respectively. 
Our Assumption~\ref{ass:speeds} states that~$\underline{c}_{p^*}<\overline{c}_{p^*}$. 
According to Weinberger~\cite{W02}, these quantities coincide with the spreading speeds 
for~\eqref{eq:parabolic} in the ranges between $p^*$ and $p_+$ and between
$p_-$ and $p^*$ respectively.
Namely, taking a constant $\delta>0$ such that
$ p^*+\delta < p_+$, and considering the Heaviside-type function
$$H(y,z):=\begin{cases} p^*(y)+\delta & \text{if }z<-\sqrt N,\\
 p^* (y) & \text{if }z\geq-\sqrt N,
\end{cases}$$
we have that for any $Z \in \mathbb{R}$, the solution $v(t,y;H(x,Z+x\cdot e))$ of \eqref{eq:parabolic} spreads with speed $\overline{c}_{p^*}$ in the following sense: for any $\varepsilon >0$, 
$$\lim_{t \to+\infty} \sup_{y \cdot e \leq (\overline{c}_{p^*}-\varepsilon) t} 
|v(t,y;  H(x,Z+x\cdot e)) - p_+ (y)| = 0,$$
$$\lim_{t \to +\infty} \sup_{y \cdot e \geq (\overline{c}_{p^*} + \varepsilon) t} 
|v(t,y;H (x,Z+x\cdot e)) - p^*(y)| = 0.$$
A similar result holds when looking at solutions between $p_-$ and $p^*$.

Let us show that $c^* \geq \overline{c}_{p^*}$. 
Since $U^* (\cdot, - \infty) \equiv \bar p \geq p_+$ and $U^* \geq p^*$, we can choose $Z >0$ large enough so that
$$U^*  \geq H (\cdot, \cdot + Z).$$
Now we argue by contradiction and assume that $c^* < \overline{c}_{p^*}$.
Then, calling $\varepsilon:=(\overline{c}_{p^*}-c^*)/2$, we have that
$\overline{c}_{p^*}-\eps=c^*+\eps$ and thus, by comparison,
\begin{equation}\label{eq:spread11}
\liminf_{n \to +\infty} \inf_{y\cdot e \leq (c^*+\varepsilon)n} 
\big(v(n,y; U^* (x,x\cdot e)) - p_+ (y )\big) \geq 0.
\end{equation}
Consequently, because
$$v(n,y ; U^* (x,x\cdot e)) = (\mathcal{F}_{e,c^*})^n [U^*] (y  , y\cdot e - nc^*) = U^* (y, y \cdot e - nc^*),$$
we find that $U^*(y,y\cdot e - nc^*) > p_+ (y)-\delta$
for $n$ sufficiently large and for all $y$ such that 
$y \cdot e \leq (c^*+ \varepsilon)n$.
Taking for instance $y = (c^* + \varepsilon )n\, e$
and passing to the limit as $n \to +\infty$ yields
$p^* (y_\infty) \geq p_+ (y_\infty)-\delta$, where $y_\infty$ 
is the limit of $(c^* + \varepsilon )n\, e$ (up to subsequences and modulo the periodicity;
recall that the limit $U^*(\cdot ,+\infty)$ is uniform). This is impossible because
$\delta$ was chosen in such a way that
$p^* + \delta < p_+$. As announced, there holds $c^* \geq \overline{c}_{p^*}$.

Let us now show that $c^* \leq \underline{c}_{p^*}$. 
The strategy is to follow a level set between $p_-$ and~$p^*$ of a suitable iteration
$(\mathcal{F}_{e,c^*})^n[a_{c,n(c)}]$ and to pass again to the (relaxed) limit 
as $c\nearrow c^*$. Notice that, in the situation where~$p$ (coming from 
Lemma~\ref{lem:limit1}$(vi)$) satisfies $p < p^*$, then it would be sufficient to consider the sequence 
$((\mathcal{F}_{e,c^*})^n [a^*])_n$ to capture such a level set; 
however it may happen that $p \equiv p^*$ and for this reason we need to 
come back to the family~$a_{c,n(c)}$.

For $k\in\N$, we can find $n_k\in\N$ such that the following properties hold:
$$
\max_{y\in[0,1]^N}\big|(\F{c^*})^{n_k}[a^*](y,k+y\cdot e)-U^*(y,k+y\cdot e)\big|
<\frac1k,
$$
$$
\max_{y\in[0,1]^N}\big|(\F{c^*})^{n_k}[a^*](y,2k+y\cdot e)-U^*(y,2k+y\cdot e)\big|
<\frac1k.
$$
Then, recalling the definition \eqref{def:a*}
of $a^*$ and up to extracting a subsequence of the sequence $c^k\nearrow c^*$ 
appearing there,
we find that for every $k\in\N$, there holds
\Fi{eq:k+1}
\max_{y\in[0,1]^N}\big((\F{c^*})^{n_k}[a_{c^k,n(c^k)}](y,k+1+z_{c^k,n(c^k)}+y\cdot e)-
U^*(y,k+y\cdot e)\big)<\frac2k,
\Ff
\Fi{eq:2k-1}
\min_{y\in[0,1]^N}\big((\F{c^*})^{n_k}[a_{c^k,n(c^k)}](y,2k-1+z_{c^k ,n(c^k)}+y\cdot e)-
U^*(y,2k+y\cdot e)\big)>-\frac2k.
\Ff
Notice that in~\eqref{eq:k+1} and \eqref{eq:2k-1} we have translated by $z_{c^k, n(c^k)}\pm1$ instead of $z_{c^k,n(c^k)}$ 
because of the `relaxed'
limit in \eqref{def:a*}. 
In order to pick the desired level set, take a constant $\delta>0$ small enough so that
$p_- +\delta< p^*$. We then define
$$\hat z_k :=\inf\{z\ :\ (\F{c^*})^{n_k}[a_{c^k,n(c^k)}](y,z+y\cdot e)- p_- (y)\leq \delta \text{ for all }y\in[0,1]^N\}.$$
Observe that $\hat z_k\in\R$ and actually $\hat z_k\geq z_{c^k ,n(c^k)}$,
as a consequence of the definition of $z_{c,n}$ in Lemma~\ref{lem:n(c)} 
and the fact that $\varphi (\cdot, - \infty) > p^*$, since it lies in the basin of attraction of $\bar p$. Because $U^*(y,+\infty)=p^*(y)>p_- (y)+\delta$
uniformly in $y$, we deduce from \eqref{eq:2k-1} that, for $k$ large enough,
$$\min_{y\in[0,1]^N}\big((\F{c^*})^{n_k}[a_{c^k,n(c^k)}](y,2k-1+z_{c^k ,n(c^k)}+y\cdot e)-
p_- (y)\big)>\delta,$$
whence $\hat z_k\geq2k-1+z_{c^k ,n(c^k)}$. It then follows from \eqref{eq:k+1} that, 
for $k$ sufficiently large,
\Fi{eq:-k+2}
\max_{y\in[0,1]^N}\big((\F{c^*})^{n_k}[a_{c^k,n(c^k)}](y,\hat z_k-k+2+y\cdot e)-
U^*(y,k+y\cdot e)\big)<\frac2k.
\Ff

We now apply Lemma \ref{lem:diagonal} to the sequence 
$\big((\F{c^*})^{n_k}[a_{c^k,n(c^k)}](y,z+\hat z_k +y\cdot e)\big)_{k\in\N}$.
This provides us with a function $\hat \alpha^*(y,z)$ 
periodic in $y$ and nonincreasing in $z$
and such that $\hat \alpha^*(y,z+y\cdot e)$ is uniformly continuous 
in $y\in\R^N$, uniformly with respect to $z\in\R$.
Moreover, proceeding exactly as in the proof of Lemma \ref{lem:limit1},
we deduce from the inequality $\F{c^k}\circ(\F{c^*})^{n_k}[a_{c^k,n(c^k)}]\geq 
(\F{c^*})^{n_k}[a_{c^k,n(c^k)}]$ that
$(\mathcal{F}_{e,c^*})^n [\hat \alpha^*]$ is nondecreasing with respect to~$n$.
The choice of $\hat z_k$ further implies that
\Fi{alpha>p+d}
\forall z< 0, \quad \max_{y \in [0,1]^N}
\Big(\hat \alpha^*(y, z + y \cdot e) -p_- (y)\Big)\geq\delta,
\Ff
$$\max_{y \in [0,1]^N}\Big(\hat \alpha^*(y,y\cdot e) - p_- (y)\Big)\leq\delta.$$ 
Finally, property \eqref{eq:-k+2} and the monotonicity in $z$ yield $\hat \alpha^*\leq p^*$.

We are now in a position to prove that $c^* \leq \underline{c}_{p^*}$. We again use a comparison argument with an Heaviside-type function. Indeed, from the above, we 
know that $\hat \alpha^* \leq \hat{H}$, where
$$\hat{H} (y,z) :=
\left\{
\begin{array}{ll}
p^* (y) & \mbox{if } z \leq \sqrt{N},\\
p_- (y)+\delta & \mbox{if } z > \sqrt{N}.
\end{array}
\right.
$$
According to Weinberger's spreading result in \cite{W02}, the solution $v(t,y;\hat{H}(x,x\cdot e))$ of~\eqref{eq:parabolic} spreads with speed~$\underline{c}_{p^*}$, which implies in
particular that for any $\varepsilon >0$,
$$\lim_{t \to \infty} \sup_{y \cdot e \geq (\underline{c}_{p^*} + \varepsilon) t} |v(t,y;\hat{H} (x,x\cdot e)) - p_- (y)| = 0.$$
By comparison we obtain
$$\limsup_{n \to +\infty} \sup_{y \cdot e \geq (\underline{c}_{p^*} + \varepsilon) n} 
\Big(v(n,y; \hat \alpha^* (x,x\cdot e)) - p_- (y)\Big)\leq 0.$$
However, because $(\mathcal{F}_{e,c^*})^n[\hat \alpha^*] $ is
nondecreasing in $n$, we have that
$$\hat \alpha^* (y,y\cdot e - n c^*)\leq
(\mathcal{F}_{e,c^*})^n [\hat \alpha^*] (y,y\cdot e - n c^*)=
v(n,y; \hat \alpha^*(x,x+\cdot e)),$$
and thus
$$\limsup_{n \to +\infty} \sup_{y \cdot e \geq (\underline{c}_{p^*} + \varepsilon) n} 
\Big(\hat \alpha^* (y,y\cdot e - n c^*) - p_- (y)\Big)\leq 0.$$
For $y \in[0,1]^N$ and $\xi_n:=[(\underline{c}_{p^*} + \varepsilon) n+\sqrt{N}]e$
there holds $(y+\xi_n)\cdot e
\geq (\underline{c}_{p^*} + \varepsilon) n$, whence
$$\limsup_{n \to +\infty}\max_{y\in[0,1]^N}
\Big( \hat \alpha^* (y+\xi_n, (\underline{c}_{p^*} + \varepsilon - c^*)n) - 
p_- (y+\xi_n)\Big) \leq 0.$$
By periodicity, we can drop the $\xi_n$ in the above expression.
We eventually deduce from~\eqref{alpha>p+d} that 
$\underline{c}_{p^*} + \varepsilon - c^*\geq0$, that is, 
$\underline{c}_{p^*} \geq c^*$ due to the arbitrariness of $\eps>0$.

In the end, we have shown that $\overline{c}_{p^*} \leq c^* \leq \underline{c}_{p^*}$, which directly contradicts Assumption~\ref{ass:speeds}. 
Lemma~\ref{lem:U1} is thereby proved.
\end{proof}
\begin{rmk}
Under the bistable Assumption~\ref{ass:bi}, obviously $p^*$ has to be 0, and therefore we have constructed a discrete travelling front connecting $\bar p$ to $0$. In order to conclude the proof of Theorem~\ref{th:bi}, one may directly skip to Section~\ref{sec:continuous}.
\end{rmk}

\section{A (discrete) propagating terrace}\label{sec:terrace}

At this stage we have constructed the `highest floor' of the terrace.
Then in the bistable case we are done. In the multistable case it remains to construct the lower floors, and thus we place ourselves under the pair of Assumptions~\ref{ass:multi} and \ref{ass:speeds}.
To proceed, we iterate
the previous argument to the restriction of \eqref{eq:parabolic} to the `interval' $[0,p^*]$, with~$p^*$ given by Lemma~\ref{lem:U1}, 
and we find a second travelling front connecting $p^*$ to another stable state smaller than $p^*$. 
For this the stability of $p^*$ is crucial.
The iteration ends as soon as we reach the $0$ state, 
which happens in a finite number of steps because there is a finite number of stable periodic steady states.

This procedure provides us with some finite sequences $(q_j)_{0\leq j\leq J}$ and 
$(U_j)_{1\leq j\leq J}$, where the $q_j$ are {\em linearly}
stable periodic steady states and the $U_j$ are discrete travelling fronts connecting $q_{j-1}$ to~$q_j$. 
We need to show that the speeds are ordered, so that the family of travelling fronts we construct is a (at this point, discrete) propagating terrace. 
It is here that we use the linear stability hypothesis in Assumption~\ref{ass:multi}.
As we mentioned in the introduction, the order of the speeds is a crucial property of the terrace, 
which is not a mere collection of unrelated fronts but what should actually emerge 
in the large-time limit of solutions of the Cauchy problem.

\begin{prop}
	Under Assumptions~\ref{ass:multi} and~\ref{ass:speeds}, the speeds $c_j$ of the fronts $U_j$ are ordered:
		$$c_1 \leq c_2 \leq \cdots \leq c_J .$$
\end{prop}

\begin{proof}	
We only consider the two uppermost travelling fronts $U_1$ and $U_2$ and 
we show that $c_1 \leq c_2$. The same argument applies for the subsequent speeds. 

We first come back to the family $(a_{c,n})_{c,n}$ and the function $a^*$ used to construct the front~$U_1$
	connecting $q_0 \equiv \bar p$ to~$q_1$. The main idea is that, capturing another level 
	set between $q_2$ and $q_1$, we should obtain a solution moving with a
	speed larger than or equal to $c_1$, but which is smaller than $q_1$.
	Then, comparing it with the second front $U_2$, we expect to recover the desired 
	inequality $c_1 \leq c_2$.

	In the proof of Lemma \ref{lem:U1}, we have constructed two sequences
	$(n_k)_{k\in\N}$, $(c^k)_{k\in\N}$, with $c^k\nearrow c_1$,
	such that \eqref{eq:k+1}, \eqref{eq:2k-1} hold with $c^*=c_1$ and $U^*=U_1$.
Take a small positive constant $\delta$ so that $q_j \pm \delta$ lie in the basin of attraction of $q_j$, 
for $j=1,2$, and moreover $\min (q_1 - q_2) \geq 2 \delta$.
Then define
	$$\hat{z}_k:=\inf\{z\ :\ (\F{c_1})^{n_k}[a_{c^k,n(c^k)}](y,z+y\cdot e)- q_2 (y)\leq
	\delta \text{ for all }y\in[0,1]^N\}.$$
The inequality
\eqref{eq:2k-1} implies that, for $y\in[0,1]^N$ and $z\leq z_{c^k,n(c^k)} + 2k -1$, there holds
$$(\F{c_1})^{n_k}[a_{c^k,n(c^k)}](y,z+y\cdot e)>U_1(y,2k+y\cdot e)-\frac2k\to q_1(y)
\quad\text{as }\;k\to+\infty.$$
Because $q_1>q_2 +\delta$, we infer that, for $k$ large enough,
$$\hat{z}_k\geq z_{c^k,n(c^k)} + 2k -1,$$
whence, by \eqref{eq:k+1},
\Fi{zk-k}
\max_{y\in[0,1]^N}\big((\F{c_1})^{n_k}[a_{c^k,n(c^k)}](y, \hat{z}_k +2-k+y\cdot e)
- U_1(y,k + y \cdot e )\big)<\frac{2}{k}.
\Ff
We now consider the sequence of functions 
$\big((\F{c_1})^{n_k}[a_{c^k,n(c^k)}](y,z+\hat{z}_k+y\cdot e)\big)_{k\in\N}$ and apply Lemma~\ref{lem:diagonal}. We obtain a function $\hat{\alpha} (y,z)$ which is periodic in $y$, nonincreasing in $z$. Moreover, it is such that $\hat{\alpha} (y,z+y\cdot e)$ is uniformly continuous in $y$, uniformly with respect to $z$, and $(\mathcal{F}_{e,c_1})^{n} [\hat \alpha]$ is nondecreasing with respect to $n$.
Our choice of $\hat{z}_k$ further implies
\begin{equation}\label{alphap'_norm}
\forall z< 0, \ \max_{ y \in [0,1]^N} \Big( \hat \alpha(y, z + y \cdot e)  - q_2 (y) \Big) \geq   \delta,
\end{equation}
\begin{equation}\label{alphap'_norm2}
\forall y \in [0,1]^N, \quad \hat \alpha (y,z+y\cdot e)\leq q_2 (y)+\delta .
\end{equation}
The latter property, together with the facts that $(\mathcal{F}_{e,c_1})^n [\hat \alpha](\cdot ,+\infty)$
is nondecreasing in $n \in \N$ and that $q_2+\delta$ lies in the basin of attraction of $q_2$, yield 
$$\hat \alpha (\cdot , + \infty) \leq q_2 .$$
On the other hand, using \eqref{zk-k} one infers that
$$\hat \alpha (\cdot , - \infty) \leq q_1 .$$

	Our aim is to compare $\hat \alpha$ with $U_2$ using the sliding method. To this end, we shall increase
	$U_2$ a bit without affecting its asymptotical dynamics, exploiting the linear stability of 
	$q_1$, $q_2$.
	Let $\varphi_{q_1}$ and $\varphi_{q_2}$ denote the periodic principal eigenfunctions associated
	with the linearization of \eqref{eq:parabolic} around $q_1$ and $q_2$ respectively,
	normalized by $\max \varphi_{q_1} =\max \varphi_{q_2}=1$. Then consider
	a smooth, positive function $\Phi=\Phi(y,z)$ which is periodic in $y$ and 
	satisfies
	$$\Phi(y,z)=\begin{cases}
				\varphi_{q_1}(y) & \text{if }z\leq-1 , \vspace{3pt}\\
				\varphi_{q_2}(y) & \text{if }z\geq1 ,
				\end{cases}$$
	and define, for $\eps \in (0,\delta)$,
	$$U_{2,\eps} (y,z):=U_2(y,z)+\eps\,\Phi(y,z).$$
	Now, because the limits as $z\to\pm\infty$ satisfy the following inequalities uniformly in $y$:
	$$U_{2,\eps}(y,-\infty)>q_1(y)\geq
	\hat \alpha(y,-\infty),\qquad 
	U_{2,\eps}(y,+\infty) > q_2(y)\geq
	\hat \alpha(y,+\infty),$$ 
	and using also \eqref{alphap'_norm},  
	we can define the following real number:
	$$Z_\eps:=\sup\big\{Z\ :\ 
		U_{2,\eps} (y,Z+z+y\cdot e)>\hat \alpha(y,z+y\cdot e)
		\text{ for all }(y,z)\in\R^{N+1}\big\}.$$ 
	Let us assume by way of contradiction that the speed of $U_2$ satisfies
	$c_2<c_1$. 
	Then if we fix
	$$\tilde Z_\eps\in(Z_\eps-c_1+c_2,Z_\eps),$$
	we can find $(y_\eps,z_\eps)\in\R^{N+1}$ such that
	\Fi{Zeps}
	U_{2,\eps} (y_\eps,\tilde Z_\eps+c_1-c_2+z_\eps+y_\eps\cdot e)\leq
	\hat \alpha(y_\eps,z_\eps+y_\eps\cdot e).
	\Ff
	Consider the following functions: 
	$$u_\eps(t,y):=v(t,y;\hat \alpha(x,c_1+z_\eps+x\cdot e)),$$
	$$w_\eps(t,y):=v(t,y;U_2 (x,\tilde Z_\eps+c_1+z_\eps+x\cdot e))
	+\eps\,\Phi(y,\tilde Z_\eps+c_1-c_2 t+z_\eps+y\cdot e).$$
	We find from one hand that
	\[\forall y\in\R^N,\quad
	w_\eps(0,y)-u_\eps(0,y)=
	U_{2,\eps}(y,\tilde Z_\eps+c_1+z_\eps+y\cdot e)
	-\hat \alpha (y,c_1+z_\eps+y\cdot e)>0\]
	because $\tilde Z_\eps<Z_\eps$. Hence, recalling that
	$U_{2,\eps},\hat \alpha$ are periodic in $y$ and satisfy
	$U_{2,\eps}(y,\pm\infty)>\hat \alpha(y,\pm\infty)$,
	which yields
	\Fi{w-u}
	\liminf_{y\cdot e\to\pm\infty}(w_\eps-u_\eps)(0,y)\geq
	\eps\min\left\{\min\varphi_{q_1},\min\varphi_{q_2}\right\}>0,
	\Ff
	we infer that
	$\inf_y(w_\eps(0,y)-u_\eps(0,y))>0$.	
	Then, by uniform continuity,
	$w_\eps>u_\eps$ for~$t>0$ small enough.
	From the other hand, using the fact that, for all $m\in\N$,
	\Fi{ueps}
	u_\eps(m,y)=\F{c_1}[\hat \alpha]^m (y,c_1+z_\eps+y\cdot e-m c_1)
	\geq \hat \alpha(y,c_1+z_\eps+y\cdot e-m c_1),
	\Ff
	\Fi{weps}
	w_\eps(m,y)=U_{2,\eps}(y,\tilde Z_\eps+c_1+z_\eps+y\cdot e-m c_2),
	\Ff
	we derive
	\[
	w_\eps(1,y_\eps)-u_\eps(1,y_\eps)\leq U_{2,\eps}(y_\eps,\tilde Z_\eps+c_1-c_2+z_\eps+y_\eps\cdot e)
	-\hat \alpha(y_\eps,z_\eps+y_\eps\cdot e),
	\]
	which is nonpositive by \eqref{Zeps}. Let us point out that, if $w_\eps$ was a supersolution on the whole domain, this would contradict the comparison principle; unfortunately we shall see below that we only know it to be a supersolution in some subdomains. Therefore we shall first use a limiting argument as $\eps \to 0$ to find that $\hat{\alpha}$ also lies below a shift of~$U_2$ itself, so that the comparison principle will become available.

	From the above we deduce the existence of a time $T_\eps\in(0,1]$ such that
	$w_\eps>u_\eps$ for $t\in[0,T_\eps)$ and
	$\inf_y(w_\eps-u_\eps)(T_\eps,y)=0$.
	There exists then a sequence $(y_\eps^n)_{n\in\N}$ satisfying $(w_\eps-u_\eps)(T_\eps,y_\eps^n)\to0$
	as $n\to+\infty$. 
	We observe that the sequence $(y_\eps^n\cdot e)_{n\in\N}$ is necessarily bounded because 
	the inequalities \eqref{w-u} hold true for all times,
	as a consequence of the fact that,
	for solutions of parabolic equations such as \eqref{eq:parabolic}, the property of being bounded 
	from one side by a steady state at the limit in a given direction is preserved along evolution.
	
	The linear stability of $q_1$ and $q_2$ means that 
	the periodic principal eigenvalues $\lambda_{q_1}$, $\lambda_{q_2}$ of
	the associated linearized operators are negative.
	Then, for a given solution $u$ to~\eqref{eq:parabolic}, the function 
	$u+\eps\varphi_{q_j}$, with $\eps>0$ and $j=1,2$, satisfies for $t>0$, $x\in\R^N$,
	\[\begin{split}
	\partial_t(u+\eps\varphi_{q_j}) - \text{div} (A(x) \nabla (u+\eps\varphi_{q_j}) )
	&=f (x,u)+(f_u(x,q_j)-\lambda_{q_j})\eps\varphi_{q_j}\\
	&=f(x,u+\eps\varphi_{q_j})+(f_u(x,q_j)-f_u(x,s)-\lambda_{q_j})\eps\varphi_{q_j},
	\end{split}\]
	for some $u(t,x)<s<u(t,x)+\eps\varphi_{q_j}(x)$.
	Thus, because $\lambda_{q_j}<0$, the regularity of~$f_u$
	allows us to find $\gamma >0$ such that
	$u+\eps\varphi_{q_j}$ is a supersolution to \eqref{eq:parabolic} 
	whenever $|u-q_j|<\gamma$ and $\eps\in(0,\gamma)$.
	From now on, we restrict to $\eps\in(0,\gamma)$.
	Take $Z\geq1$ in such a way that
	$$U_2 (\cdot,z)> q_1 -\gamma \ \text{ if }z\leq-Z,\qquad
	U_2 (\cdot,z)< q_2+\gamma \ \text{ if }z\geq Z,$$
	as well as, for all $0 \leq t \leq 1$,
	$$v (t,y; U_2 (x, \tilde Z_\eps + c_1 + z_\eps + x \cdot e)) > q_1 - \gamma 
	\quad \text{if } y \cdot e \leq -Z  - \tilde{Z}_\eps - c_1 + c_2 t - z_\eps ,$$
	$$v (t,y; U_2 (x, \tilde Z_\eps + c_1 + z_\eps + x \cdot e)) < q_2 + \gamma 
	\quad \text{if }  y \cdot e 
	\geq Z - \tilde{Z}_\eps - c_1 + c_2 t - z_\eps .$$
	We have just seen that these conditions imply the property that $w_\eps$ is a supersolution
	to~\eqref{eq:parabolic} in corresponding subdomains.
	We claim that this implies that
	\Fi{yeps}
	\liminf_{n\to+\infty}
	|\tilde Z_\eps+c_1+z_\eps+y_\eps^n\cdot e|\leq Z+|c_2|+3\sqrt{N},
	\Ff
	which will in turn guarantee that functions $u_\eps$ and $w_\eps$ do not become trivial as $\eps \to 0$.

	To prove \eqref{yeps}, consider $(k^n)_{n\in\N}$ in $\Z^N$ such that
	$y_\eps^n-k^n\in[0,1]^N$. 
	Clearly, $(k^n\cdot e)_{n\in\N}$ 	is bounded because $(y_\eps^n\cdot e)_{n\in\N}$ is.
	Let $y^\infty_\eps$ be the limit of (a subsequence of) $(y_\eps^n-k^n)_{n\in\N}$.
	The functions 
	$w_\eps(t,y+k^n)$ and $u_\eps(t,y+k^n)$ 
	converge as $n\to+\infty$ 	(up to subsequences) locally uniformly in 
	$[0,1)\times\R^N$ to some functions $\tilde w_\eps$, $\tilde u_\eps$
	satisfying
	$$\min_{[0,T_\eps]\times\R^N}(\tilde w_\eps-\tilde u_\eps)=
	(\tilde w_\eps-\tilde u_\eps)(T_\eps,y^\infty_\eps)=0.$$
	The function $\tilde u_\eps$ is a solution to \eqref{eq:parabolic}.
	Instead, $\tilde w_\eps$ is a supersolution to 
	\eqref{eq:parabolic} for $t\in(0,T_\eps]$ and $y\cdot e<2\sqrt{N}$
	or $y\cdot e>-2\sqrt{N}$ 
	if respectively one or the other of the following inequalities holds for 
	infinite values of $n$:
	$$\tilde Z_\eps+c_1 +z_\eps+k^n\cdot e<-Z-|c_2|-2\sqrt{N},
	\qquad \tilde Z_\eps+c_1 +z_\eps+k^n\cdot e>
	Z+|c_2|+2\sqrt{N}.$$	
	Hence if \eqref{yeps} does not hold we have that $\tilde w_\eps$ is a supersolution of 
	\eqref{eq:parabolic} in a half-space orthogonal to $e$ containing the point $y^\infty_\eps$, 
	and thus the parabolic strong maximum principle yields 
	$\tilde w_\eps\equiv\tilde u_\eps$ in such half-space for $t\leq T_\eps$.
	This is impossible because, by the boundedness of $(k^n\cdot e)_{n\in\N}$, 
	the property \eqref{w-u} holds true with $w_\eps-u_\eps$ replaced by~	$\tilde w_\eps-\tilde u_\eps$. This proves~\eqref{yeps}.
	 
	Using \eqref{yeps} we can find a family $(\tilde y_\eps)_{\eps\in(0,\gamma)}$ 
	such that $(\tilde Z_\eps+c_1+z_\eps+\tilde y_\eps\cdot e)_{\eps\in(0,\gamma)}$
	is bounded and
	$(w_\eps-u_\eps)(T_\eps,\tilde y_\eps)\to0$ as $\eps\to0$.
	Arguing as before, by considering the translations 
	$u_\eps(t,y+k_\eps)$, $w_\eps(t,y+k_\eps)$ with	
	$k_\eps\in\Z^N$ such that $\tilde y_\eps-k_\eps\in[0,1]^N$, we obtain at the limit
	$\eps\searrow0$ (up to some subsequences) two functions
	$\tilde u$ and $\tilde w$ which are now both solutions to \eqref{eq:parabolic} and 
	satisfy
	$$\min_{y\in\R^N}(\tilde w-\tilde u)(\tilde T,y)=
	(\tilde w-\tilde u)(\tilde T,\tilde y)=0,$$
 	where $\tilde T=\lim_{\eps\to0}T_\eps$ and $\tilde y=\lim_{\eps\to0}(\tilde y_\eps-k_\eps)$. 
	If $\tilde T>0$ then $\tilde w\equiv\tilde u$, otherwise we can only infer that
	$\tilde w\geq\tilde u$ for all times and that $(\tilde w-\tilde u)(0,\tilde y)=0$.
	In both cases, roughly the spreading speed of $\tilde{u}$ has to be less than that of $\tilde{w}$, which ultimately will contradict the inequality $c_2 < c_1$.

	More precisely, since $(\tilde Z_\eps+c_1+z_\eps+k_\eps\cdot e)_{\eps\in(0,\gamma)}$
	is bounded, we derive 
	$$\tilde u(0,\tilde y)=\tilde w(0,\tilde y)=
	\lim_{\eps\to0}w_\eps(0,\tilde y +k_\eps)=
	\lim_{\eps\to0}U_2 (\tilde y,\tilde Z_\eps+c_1+z_\eps+k_\eps\cdot e+\tilde y \cdot e),$$
	and thus $q_2 (\tilde y)<\tilde u(0,\tilde y)<q_1(\tilde y)$
	because $q_2<U_2<q_1$ thanks to Proposition \ref{prop:Fec}$(ii)$.
	Next, fix $c' \in(c_2,c_1)$ and consider a sequence $(h_m)_{m\in\N}$ satisfying
	$c'  m<h_m\cdot e<c_1 m$ for~$m$ larger than some~$m_0$. 
	From one hand, using \eqref{ueps} and the monotonicity of $\hat \alpha$ with respect to its second variable, we get
	$$\forall m\geq m_0,\ y\in\R^N,\quad
	u_\eps(m,y+h_m)\geq \hat \alpha ( y,c_1+z_\eps+(y+h_m)\cdot e-mc_1)\geq
	u_\eps(0,y),$$
	from which we deduce 
	\begin{equation}\label{eq:u>q2}
	\tilde u(m,\tilde y+h_m)\geq\tilde u(0,\tilde y)>q_2(\tilde y).
	\end{equation}
	From the other hand, \eqref{weps} yields
	$$\forall m\geq m_0,\quad	
	w_\eps(m,\tilde y+h_m+k_\eps)\leq 
	U_{2,\eps}(\tilde y,\tilde Z_\eps+c_1 +z_\eps+(k_\eps+ \tilde y) \cdot e+m(c'-c_2)),$$
	whence, letting $L>0$ be such that $\tilde Z_\eps+c_1+z_\eps+k_\eps\cdot e\geq-L$
	for all $\eps \in (0,\gamma)$, we find that
	$$\tilde w(m,\tilde y+h_m)\leq 
	U_2(\tilde y,-L+\tilde y \cdot e+m(c'-c_2)).$$
	The above right-hand side converges to $q_2(\tilde y)$ as $m\to+\infty$,
	and therefore, by \eqref{eq:u>q2},
	we derive for $m$ sufficiently large,
	$$\tilde w(m,\tilde y +h_m)<\tilde u(0,\tilde y)
	\leq \tilde u(m(z),\tilde y +h_m).$$
	This contradicts the inequality $\tilde w\geq\tilde u$, concluding the proof of the proposition.
\end{proof}

\section{To the continuous case}\label{sec:continuous}

In this section, we place ourselves under Assumption~\ref{ass:speeds} and either 
Assumption~\ref{ass:bi} or~\ref{ass:multi}. 
In both situations, we have constructed in the previous sections a `discrete' travelling front 
or terrace (i.e., a finite and appropriately ordered sequence of discrete travelling fronts)
in the sense of Definition~\ref{def:discrete_front}. 
Clearly our argument may be performed with any positive time step (not necessarily equal to 1), and thus we can consider a sequence of `discrete' terraces associated with the time steps $2^{-k}$,
$k\in\N$. 
By passing to the limit as $k \to +\infty$, we expect to recover an actual propagating terrace in the sense of Definition~\ref{def:terrace}.
\begin{rmk}
As we mentioned earlier, in some cases this limiting argument is not needed. Indeed, it is rather straightforward to show that a discrete travelling front, regardless of the time step, is also a generalized transition front in the sense of Berestycki and Hamel~\cite{BH12}; without going into the details, we recall that a transition front is an entire solution whose level sets remain at a bounded distance uniformly with respect to time. Under an additional  monotonicity assumption on the neighborhood of limiting stable steady states, and provided that the speed is not 0, they have proved that any almost planar transition front is also a pulsating travelling front. However, this is not true in general, therefore we proceed with a different approach.
\end{rmk}
For any direction $e \in \mathbb{S}^{N-1}$ and any $k \in \mathbb{N}$, 
the discrete terrace associated with the time step $2^{-k}$ consists of
 a finite sequence of ordered stable steady states
$$\bar p \equiv q_{0,k} > q_{1,k} > \cdots > q_{J(k),k} \equiv 0,$$
and a finite sequence of discrete travelling fronts
connecting these steady states with nondecreasing speeds.
Because the $(q_{j,k})_{0\leq j \leq J(k)}$ belong to the finite set of periodic stable steady states of \eqref{eq:parabolic}, we can extract from the sequence 
of time steps $(2^{-k})_{k\in\N}$ a subsequence 
$(\tau_k)_{k\in\N}$ along which
the family $(q_{j,\tau_k})_{0\leq j \leq J(\tau_k)}$ 
does not actually depend on $k$. 
Therefore, we simply denote it by $(q_j)_{0 \leq j \leq J}$.
Let $(U_{j,k})_{0 \leq j \leq J,\ k\in\N}$ be the corresponding fronts, i.e.,
the $U_{j,k} (y,z)$ are periodic in $y$, nonincreasing in~$z$
and satisfy
$$\forall (y,z) \in \mathbb{R}^{N+1}, \quad 
v(\tau_k,y; U_{j,k} (x, z + x \cdot e)) = U_{j,k} (y,z+y\cdot e - c_{j,k}),$$
with
$$c_{1,k} \leq c_{2,k}\leq \cdots \leq c_{J,k},$$
as well as 
$$U_{j,k} (\cdot, - \infty) \equiv q_{j-1} , \quad U_{j,k} (\cdot , + \infty) \equiv q_{j}.$$

As a matter of fact, the speeds $c_{j,k}$ 
are proportional to the time step $\tau_k$, by a factor depending on $j$.
This is the subject of the next lemma, whose proof exploits the link between the front and the spreading speed,
which is the heart of the method developed by Weinberger in~\cite{W02} 
and used in the present paper.
\begin{lem}\label{lem:speed_k1}
There exists a sequence
$$c_1 \leq c_2\leq \cdots \leq c_J$$
such that
$$ \forall j\in\{1,\dots, J\},\ k\in\N,\quad
c_{j,k}=\tau_k c_j.$$
\end{lem}
\begin{proof}
The proof amounts to showing that 
$$ \forall j\in\{1,\dots, J\},\ k\in\N,\quad
\frac{c_{j,k+1}}{c_{j,k}}=\frac{\tau_{j,k+1}}{\tau_{j,k}}.$$
We do it for $j=1$. 
Then, since the intermediate states $q_{i}$ do not depend on $k$, and because the subsequent speeds were constructed in a similar fashion, 
the case $j>1$ is analogously derived.

Let us first show that $\frac{c_{1,k+1}}{c_{1,k}}\geq \frac{\tau_{1,k+1}}{\tau_{1,k}}$.
This easily follows from our earlier construction. 
Let us consider the shifted evolution operators associated with the time steps $(\tau_k)_{k\in\N}$.
In analogy with \eqref{eq:def_mapping1}, these are defined by 
$$\mathcal{F}_{e,c,k} [V] (y,z+y\cdot e-c) := v (\tau_k,y; V(x,z+ x \cdot e)).$$
Then, for $\phi$ satisfying \eqref{ass:phi}, we define the sequence
$(a_{c,n}^k)_{n\in\N}$ through \eqref{eq:def_acn} with 
$\mathcal{F}_{e,c}$ replaced by $\mathcal{F}_{e,c,k}$.
Fix $k\in\N$ and call $\rho:=\frac{\tau_k}{\tau_{k+1}} \in \mathbb{N}^*$.
Because
$$(\mathcal{F}_{e,c,k+1})^\rho [V] (y,z+y\cdot e) = v (\rho\tau_{k+1},y; V(x,z+ x \cdot e+\rho c))
=\mathcal{F}_{e,\rho c,k} [V] (y,z+y\cdot e),$$
we find that
$a_{c,\rho}^{k+1} \geq a_{\rho c,1}^k$,  where the inequality comes from the fact that
in the time step $\tau_k$
the sequence $(a_{c,n}^{k+1})_{n\in\N}$ is `boosted' 
$\rho$ times by the function $\phi$, 
while $(a_{\rho c,n}^k)_{n\in\N}$ only once. We can readily iterate this argument to get
that $a_{c,\rho n}^{k+1} \geq a_{\rho c,n}^k$ for all $n\in\N$.
It then follows from Lemma~\ref{lem:ac2} that if $\rho c<c_{1,k}$ then $c<c_{1,k+1}$
This means that $\rho c_{1,k+1}\geq c_{1,k}$, which is the first desired inequality.

To prove the reverse inequality,
we shall use Lemma~\ref{lem:indep1} 
which asserts that $c_{1,k}$ does not depend on the choice of $\phi$ satisfying 
\eqref{ass:phi}. 
Then we choose the function generating the sequences 
$(a_{c,n}^k)_{n\in\N}$, $(a_{c,n}^{k+1})_{n\in\N}$ of a particular form.
Namely, we consider a solution $u$ of the Cauchy problem associated with
\eqref{eq:parabolic}, with a continuous periodic initial datum $u_0<\bar p$ such that $u_0-\delta$ lies in the  
basin of attraction of $\bar p$, for some constant $\delta>0$.
In particular, there exists $T>0$ such that $u(t,\cdot)>u_0$ for $t\geq T$.
We then initialize $(a_{c,n}^k)_{n\in\N}$ with a function 
$\phi$ satisfying $\phi(y,-\infty)=u(T,y)$.
It follows that $v(t,y;\phi(x,-\infty))>u_0(y)$ for all $t\geq0$, and thus,
by parabolic estimates, 
$$\forall 0\leq t\leq\tau_k,\  
y\in[0,1]^N,\quad
v(t,y;\phi(x,z+x\cdot e))>u_0(y),$$
provided $z$ is smaller than some $Z$.
Then, by the periodicity of $\phi$ and $u_0$, we get
$$\forall 0\leq t\leq\tau_k,\  
z+y\cdot e\leq Z-\sqrt{N},\quad
v(t,y;\phi(x,z+x\cdot e))>u_0(y).$$
We now initialize $(a_{c,n}^{k+1})_{n\in\N}$
with a function $\phi'$ satisfying 
$$\phi'(y,-\infty)=u(0,y),\qquad
\phi'(y,z)=0\ \text{ for }z\geq Z-\sqrt{N}-\rho|c|.$$
We deduce that
$$\forall j=1,\dots,\rho,\ (y,z)\in\R^{N+1},\quad
\phi' (y,z+y\cdot e- \rho|c|)\leq v (j\tau_{k+1},y; \phi (x, z + x \cdot e)).$$
We claim that $a_{c,\rho n}^{k+1}\leq a_{\rho c,n}^k$ for all $n\in\N$.
This property holds for $n=0$. Suppose that it holds for some $n\in\N$.
Using the property of $\phi'$, and recalling that 
$\phi\leq a_{\rho c,n}^k$, we find that
\[\begin{split}
a_{c,\rho n+1}^{k+1}(y,z+y\cdot e- c) &= 
\max \{ \phi' (y,z+y\cdot e- c), v (\tau_{k+1},y; a_{c,\rho n}^{k+1} (x, z + x \cdot e))\}\\
&\leq v (\tau_{k+1},y; a_{\rho c,n}^k(x, z + x \cdot e)).
\end{split}\]
Iterating $\rho$ times we get
\[\begin{split}
a_{c,\rho(n+1)}^{k+1}(y,z+y\cdot e- \rho c) &= 
\max \{ \phi' (y,z+y\cdot e-\rho c),\\
&\qquad\qquad v ( \tau_{k+1},y; a_{c,\rho n+\rho-1}^{k+1} (x, z + x \cdot e-(\rho-1) c))\}\\
&\leq v (\rho \tau_{k+1},y; a_{\rho c,n}^k(x, z + x \cdot e))\\
&= \mathcal{F}_{e,\rho c,k}[a_{\rho c,n}^k](y, z + y \cdot e-\rho c)\\
&\leq a_{\rho c,n+1}^k(y, z + y \cdot e-\rho c).
\end{split}\]
The claim $a_{c,\rho n}^{k+1}\leq a_{\rho c,n}^k$ is thereby proved for all $n\in\N$.
Then, as before, owing to Lemma~\ref{lem:ac2} we conclude that
$c_{1,k}\geq \rho c_{1,k+1}$.
\end{proof}
We are now in a position to conclude the proofs of Theorems \ref{th:bi} and \ref{th:multi}.
Namely, in the next lemma we show that for each
level $1 \leq j \leq J$ 
of the discrete propagating terrace one can find a continuous propagating terrace
whose fronts have the same speed $c_j$ from Lemma~\ref{lem:speed_k1}.
Then, by `merging' the so obtained 
$J$ terraces, one gets a propagating terrace 
of \eqref{eq:parabolic} connecting $\bar p$ to~$0$. 
In the bistable case, the terrace reduces to a single pulsating travelling front,
thanks to Assumptions \ref{ass:bi} and~\ref{ass:speeds}.
Instead, in the multistable case, our construction
allows the possibility that the continuous propagating terrace contains more fronts than the discrete terraces did. This is actually not true in typical situations 
(such as the already mentioned ones where the argument of Berestycki and Hamel
\cite{BH12} applies), 
but it remains unclear whether this can happen in general.
\begin{lem}
For any $1 \leq j \leq J$, there exists a propagating terrace connecting $q_{j-1}$ to~$q_j$
in the sense of Definition \ref{def:terrace}. 
Moreover, all the fronts in this terrace have the speed~$c_j$.
\end{lem}

\begin{proof}
The aim is to pass to the limit as $k \to +\infty$ in the sequence of
discrete terraces associated to the time steps $(\tau_k)_{k\in\N}$.
The first step consists in showing 
that the profiles $U_{j,k}$ converge as $k \to +\infty$. 
Due to the lack of regularity with respect to the second variable,
the limit will be taken in the relaxed sense of Lemma \ref{lem:diagonal}.

As usual, the argument is the same regardless of the choice of $j$
and then for simplicity of notation we take $j=1$. 
Beforehand, we shift $U_{1,k}$ so that
\Fi{U1k>}
\forall z<0,\quad
\min_{y\in[0,1]^N} \big(U_{1,k}(y,z+y\cdot e) - \eta (y)\big) \geq 0,
\Ff
\Fi{U1k<}
\min_{y\in[0,1]^N}\big(U_{1,k}(y,y\cdot e) -\eta(y)\big) \leq 0,
\Ff
where $q_1<\eta< \bar p$ is a given function lying in the basin of attraction of $\bar p$.
We know that~$U_{1,k}$ is a fixed point for $\mathcal{F}_{e,c_{1,k},k}$ by construction,
that is, it is a fixed point for $\mathcal{F}_{e,\tau_{k}c_1,k}$
owing to the previous lemma. Then, for $k'<k$, observing that
$$\mathcal{F}_{e,\tau_{k'}c_1,k'} = 
\left( \mathcal{F}_{e,\tau_{k}c_1,k}\right)^{\frac{\tau_{k'}}{\tau_k}},$$
where $\frac{\tau_{k'}}{\tau_k}\in\N$, 
we see that it is a fixed point for $\mathcal{F}_{e,\tau_{k'}c_1,k'}$ too.

We now apply Lemma~\ref{lem:diagonal} 
to the sequence $(U_{1,k})_{k\in\N}$. We point out that the hypothesis
there that $U_{1,k}(y,z + y\cdot e)$ is uniformly 
continuous in $y\in\R^N$, uniformly with respect to~$z$ and~$k$,
follows from parabolic estimates due to the fact that all the $U_{1,k}$ are fixed points of 
$\mathcal{F}_{e,\tau_{1}c_1,1}$.
We obtain in the relaxed limit (up to subsequences)
a function $U_1 (y,z)$ which is periodic in~$y$, nonincreasing in~$z$ and such that $U_1 (y,z + y\cdot e)$ is uniformly continuous in~$y$, uniformly with respect to $z$. Moreover,~$U_1$ satisfies the normalization
\eqref{U1k>}-\eqref{U1k<}.
Finally, by the above consideration,
it also follows from Lemma~\ref{lem:diagonal} and the continuity of the operators~$\mathcal{F}_{e,\tau_{k}c_1,k}$ in the locally uniform topology, that $U_1$ fulfils
$$\forall k\in\N,\quad
\mathcal{F}_{e,\tau_{k}c_1,k} [U_1] \equiv U_1.$$
Let $u(t,y)$ denote the solution of the problem~\eqref{eq:parabolic} 
with initial datum~$U_1 (y, y \cdot e)$. 
Then for any $k,m \in \mathbb{N}$, we have that
$$u(m\tau_{k},y) = \left( \mathcal{F}_{e,\tau_{k}c_1,k}\right)^m [U_1] (y,y \cdot e - m\tau_{k}c_1) 
= U_1 (y,y \cdot e - m\tau_{k}c_1).$$
By continuity of the solution of~\eqref{eq:parabolic} with respect to time, as well as the monotonicity of~$U_1$ with respect to its second variable, we immediately extend this inequality to all positive times, i.e.,
$$u( t,y) = U_1 (y,y \cdot e - c_1 t) .$$
In particular, $U_1 (y, y\cdot e -c_1 t)$ solves \eqref{eq:parabolic} for positive times in the whole space; by periodicity in the first variable, it is straightforward to check that it solves \eqref{eq:parabolic} for negative times too.
\begin{rmk}
We have shown above that $U_1$ is continuous with respect to both its variables, 
on the condition that $c_1 \neq 0$.
\end{rmk}

To show that $U_1 (y, y \cdot e- c_1 t)$ is a pulsating travelling front in the sense of Definition~\ref{def:puls_front}, it only remains to check that it satisfies the appropriate asymptotic. By monotonicity in the second variable, we already know that $U_1 (\cdot , \pm \infty)$ exist, and moreover these limits are periodic steady states of \eqref{eq:parabolic}. 
We further have that $U_1 (\cdot, -\infty) \geq \eta $ and $U_1 (\cdot, +\infty)\not\equiv\bar p $, because 
$U_1$ satisfies~\eqref{U1k>}-\eqref{U1k<}. Recalling that $\eta$ lies in the basin of attraction of $\bar p$, we find that
$U_1 (\cdot,-\infty) \equiv \bar p$.

Let us now deal with the limit as $z \to +\infty$. Let us call $p^*:=U_1 (\cdot, + \infty)$. This is a periodic
steady state satisfying $q_1\leq p^*<\bar p $; 
however it could happen that the first inequality is strict too. 
We claim that $p^*$~is stable. In which case, changing the normalization \eqref{U1k>}-\eqref{U1k<}
by taking $\eta<p^*$ in the basin of attraction of $p^*$,
and then passing to the limit as before, we end up with a new function $U_2$.
Because of this normalization, together with the fact that
$U_1(\cdot,+\infty)=p^*$, it turns out that $U_2$
connects $p^*$ to another steady state $p^*_2\geq q_1$.
Then, by iteration, we eventually construct a terrace connecting $\bar p$ to $q_1$.

It remains to show that $p^*$ is stable.
We proceed by contradiction and assume that this is not the case. In particular, $p^* > q_1$. 
Let $p_{+}$,~$p_{-}$ denote respectively the smallest stable periodic steady state above $p^*$
and the largest stable periodic steady state below~$p^*$, and let
$\overline{c}_{p^*}$ and $\underline{c}_{p^*}$ be the minimal speeds of fronts connecting $p_{+}$ to~$p^*$
and $p^*$ to~$p_{-}$ respectively.
By the same comparison argument as in the proof of Lemma~\ref{lem:U1}
one readily sees that the speed $c_1$ of~$U_1$ satisfies
$$c_1 \geq \overline{c}_{p^*}.$$
We recall that the argument exploits Weinberger's result in~\cite{W02}
which asserts that $\overline{c}_{p^*}$ coincides with the
spreading speed for solutions between $p^*$ and $p_{+}$.
Next, one shows that 
$$c_1 \leq \underline{c}_{p^*}.$$
This is achieved by choosing the normalization 
$$
\forall z<0,\quad
\max_{y\in[0,1]^N} \big(U_{1,k}(y,z+y\cdot e) - \eta^* (y)\big) \geq 0,
$$
$$
\max_{y\in[0,1]^N}\big(U_{1,k}(y,y\cdot e) - \eta^*(y)\big) \leq 0,$$
with $\eta^*$ between $p_{-}$ and $p^*$ and in the basin of attraction of $p_{-}$, which is possible because
$U_{1,k}(\cdot,+\infty)\equiv q_1\leq p_{-}$.
One gets in the (relaxed) limit a solution $U^* (y, y\cdot e -c_1 t)$
satisfying $U^*(\cdot,-\infty)\leq p^*$ (because compared with $U_1$, the function $U^*$ is obtained as the limit of an infinite shift of the sequence $U_{1,k}$), as well as $U^*(\cdot,+\infty)\leq \eta^*$. Then the desired inequality follows again from the spreading result.
Finally, combining the previous two inequalities 
one gets $\overline{c}_{p^*}\leq \underline{c}_{p^*}$, which 
contradicts our Assumption~\ref{ass:speeds}. This concludes the proof.
\end{proof}
\begin{rmk}
As pointed out in Remark~\ref{rmk:trap}, in general it is not equivalent to find a function~$U_1$ as above and a pulsating front solution.
The function $U_{1}$ constructed above actually gives rise to a whole family of pulsating fronts $U_1 (x, x \cdot e + z - c_1 t)$. In the case when $c_1 \neq 0$, then this family merely reduces to the time shifts of a single front. In the case when $c_1 = 0$, however, it is much less clear how these fronts are related to each other: as observed earlier the function $U_1 (x,z)$ may be discontinuous with respect to $z$, hence the resulting family may not be a continuum of fronts (in general, it is not).
\end{rmk}

%-------------------------------------------------------------------------------

\section{Highly non-symmetric phenomena}\label{sec:asymmetric}

It is clear that, because equation \eqref{eq:parabolic} is heterogeneous, the  terrace
$((q_{j})_j, (U_j)_j)$
provided by Theorem~\ref{th:multi} depends in general on the direction $e$.
In this section, we shall go further and exhibit an example where not only the fronts $U_j$, but also the
intermediate states $q_{j}$ and even their number, i.e., the number of `floors' 
of the terrace, change when~$e$ varies.
Obviously this cannot happen in the bistable case where the stable 
steady states reduce to $\bar p , 0$. Namely, we prove
Proposition~\ref{prop:asymmetric}.

The main idea is to stack a heterogeneous bistable problem below
an homogeneous one. Then in each direction there exists an ordered pair of 
pulsating travelling fronts. 
Whether this pair forms a propagating terrace depends on the order of their speeds.
If the latter is admissible for a terrace, that is, if the uppermost front is 
not faster than the lowermost,
then the terrace will consists of the two fronts, otherwise it will reduce to a single front.
Since those speeds are given respectively by a function 
$\mathbb{S}^{N-1}\ni e \mapsto c(e)$ and by a constant $c$,
and since the heterogeneity should make such a function 
nonconstant, it should be possible to end up with a case 
where the number of fronts of the terrace is nonconstant too.

Owing to the above consideration, the construction essentially amounts to finding a
heterogeneous bistable problem for which the speed of the pulsating travelling front $c(e)$
is nonconstant in $e$.
While such property should be satisfied by a broad class of problems (perhaps even generically),
getting it in the context of a bistable equation (in the sense of Assumption~\ref{ass:bi})
is rather delicate. We were not able to find an example of this type in the literature.

We place ourselves in dimension $N=2$ and denote a generic point in $\R^2$ by
$(x,y)$, as well as $e_1:=(1,0)$, $e_2:=(0,1)$. 
We derive the following.
\begin{prop}\label{pro:speeds}
	There exists a function $f_1=f_1 (y,u)$ which is periodic in the variable $y\in\R$,
	satisfies Assumptions~\ref{ass:bi},~\ref{ass:speeds} with $\bar p\equiv1$,
	and for which the equation
	\Fi{eq:f1}
	\partial_t u = \Delta u + f_1(y,u), \quad t \in \R, \ (x,y) \in \R^2 ,
	\Ff
	admits a unique (up to shifts in time) pulsating travelling front 
	connecting $1$ to $0$ for any given direction $e\in\mathbb{S}^{1}$.
	
	Furthermore, the corresponding speeds $c(e)$ satisfy $c(e_1)>c(e_2)>0$.	
\end{prop}

	The function $f_1 (y,u)$ we construct will be periodic in $y$
	with some positive period, which one can then reduce to $1$ (to be coherent with
	the rest of the paper) by simply rescaling the spatial variables.

We first introduce a smooth function $f_0:[0,1]\to\R$ with the following properties:
	$$f_0(0)=f_0\Big(\frac12\Big)=f_0(1)=0,\qquad
	f_0<0 \text{ in }\Big(0,\frac12\Big),\qquad
	f_0>0 \text{ in }\Big(\frac12,1\Big),$$
	$$ f_0'(0)=f_0'(1)=-1,\qquad
	f_0'\Big(\frac12\Big)>0,\qquad
	|f_0 '|\leq1,\qquad
	\int_0^1 f_0>0.$$
	We let $\frac12<S<1$ be the quantity identified by the relation
	$$
	\int_0^S f_0=0.
	$$
	Next, we consider two smooth functions $\chi_i:\R\to\R$,
	$i=1,2$, satisfying 
	$\chi_i\geq0,\not\equiv0$ and 
	$$\supp\chi_1\subset(0,1),\quad\chi_1=1\text{ on }\Big[\frac14,\frac34\Big],
	\qquad\supp\chi_2\subset(S,1),$$
where $\supp$ denotes the closed support.
	We then set
	\Fi{def1}
	\forall u \in\R,\ y\in[0,2L],\quad
	f_1 (y,u):=f_0(u) + M\chi_1\Big(\frac y L\Big)\,\chi_2(u),
	\Ff 
	  $L,M$ being positive constants that will be chosen later. 
	We finally extend $f_1 (y,u)$ to~$\R^2$ by periodicity
	in the $y$-variable, with period $2L$. 
	Observe that $f_1 (y,u)\geq f_0(u)$, and that equality holds for $y\in[(2j-1)L,2jL]$, $j\in\Z$.
	Until the end of the proof of Proposition~\ref{pro:speeds}, 
	when we say that a function is periodic we mean that its period is $2L$.
	
	Let us show that the equation~\eqref{eq:f1} is bistable in the sense of
	Assumption~\ref{ass:bi}. We shall also check that it fulfils 
	Assumption~\ref{ass:speeds}, for which, owing to Proposition~\ref{prop:counter}
	in the Appendix, it is sufficient to show that any intermediate state
	is linearly unstable.  
	We shall need the following observations about the periodic steady states 
	of the homogeneous equation.
	
	\begin{lem}\label{lem:f0}
		For the equation
		\Fi{eq:f0}
		\partial_t u = \Delta u + f_0(u), \quad t \in \R, \ (x,y) \in \R^2 ,
		\Ff
		the following properties hold:
		\begin{enumerate}[$(i)$]
		\item the constant steady states $0$, $1$ are linearly stable, whereas 
		$\frac12$ is linearly unstable;
		\item any periodic steady state which is not identically constant is linearly unstable;
		\item there does not exist any pair $0<q<\tilde q<1$ of periodic steady states.
		\end{enumerate}
	\end{lem}

	\begin{proof}
		Statement $(i)$ is trivial, because the principal eigenvalue of the 
		linearized operator around the constant states $q_1\equiv0$, $q_2\equiv0$,
		$q_3\equiv\frac12$ is equal to $f_0'(q_i)$.
		
		Statement $(ii)$ is a consequence of the invariance of the equation by 
		spatial translation.
		Indeed, if $q$ is a steady state which is not identically 
		constant then it admits a partial derivative $\partial_i q$ which is not identically
		equal to $0$; if in addition $q$ is periodic then
		$\partial_i q$ must change sign.
		Then, differentiating the equation $\Delta q + f_0(q)=0$ 
		with respect to~$x_i$ we find that
		$\partial_i q$ is a sign-changing eigenfunction of the linearized operator 
		around~$q$, with eigenvalue 0. 
		It follows that the principal eigenvalue $\lambda_q$
		of such operator in the space of periodic functions 
		(which is maximal, simple and associated with a positive eigenfunction)  is positive,
		that is, $q$ is linearly unstable.
		
		We prove statement $(iii)$ by contradiction. Assume that \eqref{eq:f0} admits
		a pair of periodic steady states  $0<q<\tilde q<1$. 
		We know from $(i)$-$(ii)$ that such solutions are linearly unstable.
		Then, calling $\varphi_q$  the principal eigenfunction associated with
		$\lambda_q$, one readily checks that for $\eps>0$ sufficiently small,
		$q+\eps\varphi_q$ is a stationary strict subsolution of~\eqref{eq:f0}.
		Take $\eps>0$ such that the above holds and in addition
		$q+\eps\varphi_q<\tilde q$. 
		It follows from the parabolic comparison principle
		that the solution with initial datum $q+\eps\varphi_q$ is strictly increasing in 
		time and then it converges as 
		$t\to+\infty$ to a steady state $\hat q$ satisfying $q<\hat q\leq\tilde q$.
		This is impossible, because $\hat q$ is linearly unstable by $(i)$-$(ii)$
		and then its basin of attraction
		cannot contain the function $q+\eps\varphi_q$.
	\end{proof}

	\begin{rmk}
		Consider the homogeneous equation~\eqref{eq:homo}
		with a general reaction term $f=f
		(u)$. 
		Statement $(ii)$ of Lemma~\ref{lem:f0} holds true in such case,
		because its proof only relies on the spatial-invariance of the equation. 
		Thus, if Assumption~\ref{ass:bi} holds, the uppermost steady state $\bar p$
		must be constant. One then finds that 
		Assumption~\ref{ass:bi} necessarily implies that
		$$
		\exists \theta\in(0,\bar p),\quad
		f(0)=f(\theta)=f(\bar p)=0,\quad f<0\text{ in }(0,\theta),\quad 
		f>0\text{ in }(\theta,\bar p).
		$$
		As a matter of fact, these conditions are equivalent to Assumption~\ref{ass:bi}.
		Indeed, even though the constant state $\theta$ may not
		be linearly unstable (if $f'(\theta)=0$), one sees that~$\theta$ 
		is unstable in a strong sense:
		$\theta+\eps$ belongs to the basin of attraction of $0$ if $\eps<0$ and
		of $\bar p$ if $\eps>0$.
		This is enough for the proof of Lemma~\ref{lem:f0} $(iii)$ to work. 
	\end{rmk}	

We can now derive the bistability character of \eqref{eq:f1}.

	\begin{lem}\label{lem:f1}
		Consider the equation~\eqref{eq:f1} with $f_1$ defined by \eqref{def1}.
		The following properties hold:
		\begin{enumerate}[$(i)$]
		\item	any periodic steady state $0<q<1$ is linearly unstable;
		\item there does not exist any pair $0<q<\tilde q<1$ of periodic steady states.
		\end{enumerate}
	\end{lem}

	\begin{proof}
		The proof is achieved in several steps.
		
		 \smallskip
		 \setlength{\leftskip}{\parindent}
		 \noindent 
		 {\em Step 1: any periodic steady state which is not $x$-independent is 
		 linearly unstable.}

		 \setlength{\leftskip}{0cm}
		 \noindent
			Because the equation \eqref{eq:f1} is invariant
			by translation in the $x$-variable, we can proceed exactly as in the proof of 
			Lemma~\ref{lem:f0} $(i)$.
		
		 \smallskip
		 \setlength{\leftskip}{\parindent}
		 \noindent 
		{\em Step 2: if $0\leq q<1$ is a periodic steady state which is 
		$x$-independent then $q\leq S$.}
		
		\setlength{\leftskip}{0cm}
		 \noindent	
			We recall that $S$ is defined by $\int_0^S f_0=0$. 
			Suppose that $q=q(y)$ is not constant, 
			otherwise it is identically equal to $0$ or $\frac12<S$.
			Consider an arbitrary $\eta\in\R$ with $q'(\eta)>0$.
			Let $a<\eta<b$ be such that $q'(a)=q'(b)=0$ and $q'>0$ in $(a,b)$.
			Multiplying the inequality $-q''=f_1 (y,q)\geq f_0(q)$ by $q'$ and integrating on $(a,b)$ we~get
			$$0\geq\int_a^b f_0(q(y))q'(y) dy =\int_{q(a)}^{q(b)} f_0(u)du.
			$$ 
			This implies first that $q(a)\leq\frac12$, 
			and then that $\int_{0}^{q(b)} f_0(u)du\leq0$.
			Recalling the definition of $S$, we find that $q(b)\leq S$, whence 
			$q(\eta)<S$. We have thereby shown that $q(\eta)<S$ whenever $q'(\eta)>0$,
			and therefore that $q\leq S$.
					
		\smallskip
		 \setlength{\leftskip}{\parindent}
		 \noindent 
		{\em Step 3: if \eqref{eq:f1} admits a pair of periodic steady states 
		$0<q<\tilde q<1$, then 
			there exists a periodic steady state 
			$q\leq\hat q\leq S$ which is not linearly unstable.}
		
			\setlength{\leftskip}{0cm}
		 	\noindent
		 	If $\tilde q$ is not linearly unstable then the Steps 1-2 imply that 
			$\tilde q\leq S$,
			which means that the conclusion holds
			with $\hat q=\tilde q$ in such case.
			Suppose now that $\tilde q$ is linearly unstable. It follows from 
			the same argument as in the proof of Lemma~\ref{lem:f0} $(iii)$
			that for $\eps>0$ sufficiently small, the 
			function $\tilde q-\eps\varphi_{\tilde q}$
			is a supersolution of~\eqref{eq:f1}, 
			which is larger than~$q$,
			where $\varphi_{\tilde q}$
			is the principal eigenfunction of the linearized operator around $\tilde q$. 
			The comparison principle then implies
			that the solution of \eqref{eq:f1}
			with initial datum $\tilde q-\eps\varphi_{\tilde q}$ is strictly decreasing in time and then it converges as 
			$t\to+\infty$ to a steady state $\hat q$ satisfying $q\leq\hat q<\tilde q-\eps\varphi_{\tilde q}$.
			Such state cannot be linearly unstable, because its basin of attraction contains the function
			$\tilde q-\eps\varphi_{\tilde q}$. Then, as before, $\hat q\leq S$ by the Steps 1-2.
			
		\smallskip	
		{\em Step 4: conclusion.}
			
			\noindent
			Assume by contradiction that there is a periodic steady state $0<q<1$
			which is not linearly unstable.
			From the Steps~1-2 we deduce that $q\leq S$. 
			This means that $q$ is a stationary solution of~\eqref{eq:f0}. 
			Lemma~\ref{lem:f0} then implies that $q$ is linearly unstable
			for~\eqref{eq:f0}, and thus for \eqref{eq:f1} too.
			This is a contradiction.  We have thereby proved $(i)$.
			Suppose now \eqref{eq:f1} admits a pair of periodic steady states 
			$0<q<\tilde q<1$.
			Then Step~3 provides us with a periodic steady state $0<\hat q\leq S$ 
			which is not linearly unstable, contradicting~$(i)$.		
	\end{proof}
	
Let $f_1$ be defined by \eqref{def1}, with $L,M>0$ still to be chosen.
Lemma~\ref{lem:f1} implies that the equation \eqref{eq:f1} 
is bistable in the sense of Assumption~\ref{ass:bi} with $\bar p\equiv1$.
Moreover, thanks to Proposition~\ref{prop:counter} in the Appendix, it also entails  
Assumption~\ref{ass:speeds}.
We can thus apply Theorem \ref{th:bi}, which provides us with a 
monotonic in time pulsating travelling front connecting $1$ 
to $0$, for any given direction $e\in \mathbb{S}^{1}$.	
Let~$c(e)$ be the associated speed.
Before showing that $c(e_1)>c(e_2)$, let us derive the uniqueness of the pulsating 
travelling front and the positivity of its speed.

	\begin{lem}\label{lem:uniqueness}
		The equation~\eqref{eq:f1} with $f_1$ defined by \eqref{def1}
		admits a unique (up to shifts in time) pulsating travelling front 
		connecting $1$ to $0$ for any given direction $e\in\mathbb{S}^{1}$.	
	
		Furthermore, the front is strictly increasing in time 
		and its speed $c(e)$ is positive.	
	\end{lem}
	
	\begin{proof}
		Firstly, the positivity of the speed of any front
		connecting $1$ to $0$ 
		is an immediate consequence of the facts that $f_1 \geq f_0$
		and that equation~\eqref{eq:f0} admits solutions with compactly supported initial data
		which spread with a positive speed~\cite{AW}.
		
		Next, the fronts provided by Theorem~\ref{th:bi} are monotonic in time.
		Applying the strong maximum principle to their temporal derivative (which 
		satisfies a linear parabolic equation) we infer that the
		monotonicity is strict, unless they are constant in time. The 
		positivity of their speed then implies that they are necessarily strictly increasing in time.
		Hence, the second part of the lemma holds for the fronts given by Theorem~\ref{th:bi}.
		If we show that such fronts are the only ones existing we are done.
		
		Throughout this proof, we use the notation $x$ to indicate a point in $\R^2$.
		Let $u_i(t,x)=U_i(x,x\cdot e-c_i t)$, $i=1,2$, be two
		pulsating travelling fronts for~\eqref{eq:f1} connecting $1$ to $0$
		in a given direction $e\in\mathbb{S}^{1}$. 
		We have seen before that necessarily $c_i>0$.
		This means that the transformation $(x,t)\mapsto (x,x\cdot e-c_i t)$
		is invertible and thus $U_i(x,z)$ enjoys the regularity in $(x,z)$
		coming from the parabolic regularity for $u_i$ 
		(at least $C^1$, with bounded derivatives).
		Let us suppose to fix the ideas that $c_1\geq c_2$.
		We shall also assume that 
		either $U_1$ or $U_2$ is the front provided by~Theorem \ref{th:bi},
		so that we further know that it is decreasing in $z$.	
		
		We use a sliding method.
		The conditions $U_i(\cdot,-\infty)=1$ and $U_i(\cdot,+\infty)=0$
		imply that for any $\eps\in(0,1)$, the following property holds for 
		$-k >0$ sufficiently large (depending on $\eps$):
		$$\forall x\in\R^2,\ z\in\R,\quad
		U_1(x,z)<U_2(x,z+k)+\eps.$$
		The above property clearly fails for $k>0$ large, thus we can define $k^\eps\in\R$ 
		as the supremum for which it is fulfilled. Call $U_2^\eps(x,z):=
		U_2(x,z+k^\eps)$ and $u_2^\eps(t,x):=U_2^\eps(x,x\cdot e-c_2 t)$. 
		Observe that $u_2^\eps$ is just a temporal translation of $u_2$, because $c_2\neq0$,
		whence it is still a solution of~\eqref{eq:f1}. 
		We see that  
		$$\sup(U_1-U_2^\eps)=\eps.$$
		Using again $U_i(\cdot,-\infty)=1$ and $U_i(\cdot,+\infty)=0$
		one infers that a maximizing sequence $(x_n,z_n)_{n\in\N}$
		for $U_1-U_2^\eps$ has necessarily $(z_n)_{n\in\N}$ bounded. 
		By periodicity, we can assume that the sequence $(x_n)_{n\in\N}$ 
		is contained in $[0,L]^2$. Hence, there exists $(x^\eps,z^\eps)$ such that
		$$(U_1-U_2^\eps)(x^\eps,z^\eps)=
		\max(U_1-U_2^\eps)=\eps.$$
		It follows that
		$$u_1(t^\eps_1,x^\eps)=u_2^\eps(t^\eps_2,x^\eps)+\eps,\quad\text{with }\
		t^\eps_i:=\frac{x^\eps\cdot e-z^\eps}{c_i}.$$
		Next, if $U_1(x,z)$ is the front decreasing in $z$ provided by~Theorem \ref{th:bi} then
		for $t\leq0$ we find that
		\[\begin{split}
		u_1(t^\eps_1+t,x)&=U_1(x,x\cdot e-c_1(t^\eps_1+ t))\\
		&\leq
		U_1(x,x\cdot e-c_1 t^\eps_1-c_2 t)\\
		&\leq 
		U_2^\eps(x,x\cdot e-c_2(t^\eps_2+t))+\eps=u_2^\eps(t^\eps_2+t,x)+\eps,
		\end{split}\]
		where we have used the equality $c_1 t^\eps_1=c_2 t^\eps_2$.
		Similarly, if $U_2(x,z)$ is decreasing in $z$ then, for $t\leq0$, we get
		$$u_1(t^\eps_1+t,x)\leq
		U_2(x,x\cdot e-c_2 t^\eps_2-c_1 t)+\eps\leq
		u_2^\eps(t^\eps_2+t,x)+\eps.$$
		Namely, in any case, $u_1(t^\eps_1+t,x)$ lies below 
		$u_2^\eps(t^\eps_2+t,x)+\eps$ until $t=0$, when the two
		functions touch.
		Both $u_1(t^\eps_1+t,x)$ and
		$u_2^\eps(t^\eps_2+t,x)$ are solutions of~\eqref{eq:f1}.
		Moreover, because $f_1 =f_0$ for $u$ close to $0$ and $1$, 
		and $f_0'(0),f_0'(1)<0$, one readily checks that the function 
		$u_2^\eps(t^\eps_2+t,x)+\eps$ is a supersolution of~\eqref{eq:f1}
		in the regions where it is smaller than $\delta$ or larger than $1-\delta+\eps$,
		for some small $\delta$ depending on $f_0$ and~$S$.
		If the contact point $x^\eps$ were in one of such regions,
		the parabolic strong maximum principle would imply that 
		$u_1(t^\eps_1+t,x)\equiv u_2^\eps(t^\eps_2+t,x)+\eps$ there, for $t\leq0$,
		which is impossible because $U_i(\cdot,-\infty)=1$ and $U_i(\cdot,+\infty)=0$.
		Therefore, we have that 
		$$\delta\leq u_1(t^\eps_1,x^\eps)=u_2^\eps(t^\eps_2,x^\eps)+\eps
		=u_2\Big(t^\eps_2-\frac{k^\eps}{c_2},x^\eps\Big)+\eps\leq 1-\delta+\eps.$$
		Now, because $x^\eps\in[0,L]^2$, the above bounds imply that both 
		$t^\eps_1$ and $t^\eps_2-\frac{k^\eps}{c_2}$
		stay bounded as $\eps\searrow0$. Calling $\hat x,\hat t_1,\hat t_2$ the limits
		as $\eps\searrow0$ of (some converging subsequences of) $x^\eps$, $t^\eps_1$,
		$t^\eps_2-\frac{k^\eps}{c_2}$ respectively,
		we eventually deduce that 
		$$u_1(\hat t_1,\hat x)=u_2(\hat t_2,\hat x)\qquad\text{and }\
		\forall t\leq0,\ x\in\R^2,\quad
		u_1(\hat t_1+t,x)\leq u_2(\hat t_2+t,x).$$
		The parabolic strong maximum principle finally yields 
		$u_1(\hat t_1+t,x)\equiv u_2(\hat t_2+t,x)$ for $t\in\R$, $x\in\R^2$.	
		This concludes the proof of the lemma.
	\end{proof}
	
\begin{proof}[Proof of Proposition~\ref{pro:speeds}]
		We need to show that $c(e_1)>c(e_2)$ for a suitable choice of~$L,M>0$. 
		The proof is divided into several parts.
		
		\smallskip	
		{\em Step 1: for $L>8$ there holds that $c(e_1)\to+\infty$ as
		$M\to+\infty$.}
			
		\noindent
		Fix an arbitrary $c>0$.
		We want to construct a subsolution $\underline u$ of \eqref{eq:f1} of the form
		$$\underline{u}(t,x,y):=w\circ\gamma(t,x,y),\qquad
		\gamma(t,x,y):=2-\Big|\Big(x-ct,y-\frac L2\Big)\Big|,$$
		for a suitable function $w\in W^{2,\infty}(\R)$.
		Let $S<\sigma_1<\sigma_2<1$ be such that 
		$$m:=\min_{[\sigma_1,\sigma_2]}\chi_2>0.$$
		We then define $w$ as follows:
		$$w(r):=\begin{cases}
		r^{c+3} & \text{for }0\leq r\leq \rho,\\
		r^{c+3} - \frac{Mm}{2}(r-\rho)^2 & \text{for } r>\rho,
		\end{cases}
		$$
		with $\rho:=\sigma_1^{\frac1{c+3}}$ so that $w(\rho)=\sigma_1$.	
		For $r>\rho$ we compute
		$$w'(r)=(c+3)r^{c+2}- Mm(r-\rho),$$
		which is negative for $M$ large.
		This implies that, for $M$ sufficiently large (depending on $c$),
		there exists $R_M>\rho$ such that
		$$w'>0\ \text{ in }(0,R_M),\qquad w'(R_M)=0.$$
		It also yields that $R_M\searrow\rho$ as $M\to+\infty$.
		Thus, for $M$ large enough there holds that
		$\sigma_2>R_M^{c+3}>w(R_M)$. 
		From now on we restrict ourselves to such values of $M$.
				
		Direct computation shows that
		the function $\underline u$ satisfies (in the weak sense)
		$$\partial_t \underline u - \Delta \underline u\leq
		-w''\circ\gamma+\left(c+\frac1{|(x-ct,y-\frac L2)|}\right)|w'\circ\gamma|.
		$$
		Hence, if $0<\gamma(t,x,y)<R_M$, i.e., if $2-R_M<|(x-ct,y-\frac L2)|<2$,
		recalling that $R_M<\sigma_2^{\frac1{c+3}}<1$, we get
		$$\partial_t \underline u - \Delta \underline u \leq
		-w''\circ\gamma+\left(c+1\right)w'\circ\gamma.$$
		
		Consider first the case
		$0<\gamma<\rho\ (<1)$.
		We see that
		$$w''\circ\gamma-\left(c+1\right)w'\circ\gamma=
		(c+3)(c+2)\gamma^{c+1}-(c+1)(c+3)\gamma^{c+2}
		>(c+3)\gamma^{c+2}>w\circ\gamma.$$
		Recalling that $|f_0'|\leq 1$, 
		we get $w\circ\gamma\geq-f_0(w\circ\gamma)\geq-f_1 (y,w\circ\gamma)$. 
		This means that
		$\underline u$ is a subsolution of \eqref{eq:f1} in the region $0<\gamma<\rho$.
		
		Instead, if $\rho<\gamma<R_M$, 
		there holds that $\sigma_1<w\circ\gamma<\sigma_2$ and thus
		$$w''\circ\gamma-\left(c+1\right)w'\circ\gamma>-Mm\geq-M\chi_2(w\circ\gamma).$$
		Observe that $|y-\frac L2|<2$ because $\gamma(t,x,y)>0$, whence 
		$\frac14 L<y<\frac34L$ provided $L>8$. 
		Then, under such condition, it turns out that 
		$\underline u$ is a subsolution of \eqref{eq:f1} in the region 
		$\rho<\gamma<R_M$ too.
				
		We finally extend $w$ to $0$ on $(-\infty,0)$ and we change it into the
		constant $w(R_M)\ (<\sigma_2)$ on $[R_M,+\infty)$. 
		This is still of class $W^{2,\infty}$ and, for $L>8$
		and $M$ large enough, the function 
		$\underline u:=w\circ\gamma$ is a generalized subsolution 
		of~\eqref{eq:f1} in the whole space.	
		
		Notice that $\underline u$ shifts in the direction $e_1$ with speed $c$.
		Moreover, for fixed time, it is compactly supported and bounded from above by 
		$\sigma_2$. It follows that, 
		up to translation in time, it can be placed below the pulsating travelling front
		in the direction $e_1$.
		This readily implies by comparison that the speed 
		of the latter satisfies	$c(e_1)\geq c$. 
		Step 1 is thereby proved due to the arbitrariness of $c$.
		
		\smallskip
		\setlength{\leftskip}{\parindent}
		\noindent 
		{\em Step 2: for $L>\ln4$,
			there exists $\tau>0$, depending on $L$ but not on $M$,
			such that $c(e_2)\leq2L/\tau$.}
		
		\setlength{\leftskip}{0cm}	
		\noindent
		We introduce the following function:
		$$\psi(t,y):=\frac14+e^{2t-y -L}.$$
		This is a strict supersolution of \eqref{eq:f0}. Indeed, we have that
		$$\partial_t \psi-\Delta \psi=\psi-\frac14>f_0(\psi),$$
		where the last inequality holds because $f_0(\psi)<0$ if $\frac14 < \psi<\frac12$
		and $f_0(\psi)\leq\psi-\frac12$ if $\psi > \frac12$.
		We now let $\tau$ be such that $\psi(\tau,0)=\frac12$, that is,
		$$\tau=\frac12\left(L-\ln4\right).$$
		In order to have $\tau>0$ we impose $ L>\ln4$.
		We finally define
		$$\forall j\in\N,\ t\in(0,\tau],\ y\in\R,\quad
		\bar u(j\tau+t,y):=\psi(jL+t,y).$$
		The function $\bar u(t,y)$ is increasing and 
		lower semicontinuous in $t$, because $L>\tau$.
		
		Consider now a pulsating travelling front $u(t,x,y)=U(x,y,y-c(e_2) t)$ 
		for \eqref{eq:f1} in the direction $e_2$ connecting
		$1$ to $0$. The functions $u$ and $U$ are periodic in the $x$ variable. 
		Moreover, there exists $k\in\N$ such that 
		$\bar u(k\tau,y)>U(x,y,y)=u(0,x,y)$ for all $(x,y)\in\R^2$.
		Assume by contradiction that the inequality
		$\bar u(k\tau+t,y)>u(t,x,y)$
		fails for some positive time $t$ and 
		let $T\geq0$ be the infimum of such times. Then, because $\bar u$ is increasing in 
		the first variable and $u$ is continuous,		
		we have that
		$\bar u(k\tau+t,y)\geq u(t,x,y)$ for all $t\in[0,T]$. Moreover, 
		there exist some sequences 
		$t_n\searrow T$ and $((x_n,y_n))_{n\in\N}$ such that $\bar u(k\tau+t_n,y_n)\leq
		u(t_n,x_n,y_n)$ for all $n\in\N$. By the periodicity of $u$ in $x$,
		it is not restrictive to assume that the sequence $(x_n)_{n\in\N}$ is bounded.
		The sequence $(y_n)_{n\in\N}$ is also bounded, because from one hand
		$$\bar  u(k\tau+t_n,y)\geq \bar u(k\tau+T,y)\geq\psi(T,y)>e^{2T-y -L},$$
		which is larger than $1=\sup u$ if $y<2T -L$, while from the other hand 
		$u(t,x,y)=U(x,y,y-c(e_2) t)$
		which converges to $0<\inf\bar u$ as $y\to+\infty$,
		uniformly in $x$ and locally uniformly in $t$.
		Let $(\bar x,\bar y)$ be the limit of (a converging subsequence) 
		of $((x_n,y_n))_{n\in\N}$. The continuity of $u$ and the lower semicontinuity
		of $\bar u$ yield $\bar u(k\tau+T,\bar y)\leq u(T,\bar x,\bar y)$,
		whence in particular $T>0$.
		Summing up, we have that
		\Fi{contact}
		\min_{{0\leq t\leq T}\atop{(x,y)\in\R^2}}
		\big(\bar u(k\tau+t,y)-u(t,x,y)\big)=0=
		\bar u(k\tau+T,\bar y)-u(T,\bar x,\bar y).
		\Ff
		Let $j\in\N$ be such that $k\tau+T\in(j\tau,(j+1)\tau]$.
		Using the inequalities
		\[\begin{split}
		1 &>u(T,\bar x,\bar y)=\bar u(k\tau+T,\bar y)=\psi(k \tau+ T-j\tau+jL,\bar y)=
		\frac14+e^{2(k \tau+ T-j\tau+jL)-\bar y -L}\\
		&>\frac14+e^{(2 j-1)L-\bar y},
		\end{split}\]
		we find that $\bar y>(2j-1) L$.
		
		We claim that 
		$f_1(y,\bar u(k\tau+t,y))=f_0(\bar u(k\tau+t,y))$ for $t\leq T$ and 
		$y>(2j-1) L$.
		Clearly, the claim holds if $(2j-1) L < y< 2j L$, because $f_0$ and $f_1$ coincide there.
		Take $t\leq T$ and $y\geq 2j L$.
		We see that
		$$\bar u(k\tau+t,y)\leq\frac14+e^{2(k\tau + T-j\tau+jL)-y- L}
		\leq \frac14+e^{2(\tau+jL)-y - L}
		\leq \frac14+e^{2\tau-L}=\frac12,$$
		where the last equality follows from the definition of $\tau$.
		In particular, $\bar u(k\tau+t,y)<S$ and therefore 
		$f_1(y,\bar u(k\tau+t,y))=f_0(\bar u(k\tau+t,y))$.
		This proves the claim. Thus, the function~$\psi$ being
		a strict supersolution of \eqref{eq:f0}, as seen before, we deduce that
		$\bar u$ is a (continuous) 	strict supersolution of \eqref{eq:f1}
		for $t\in(j\tau,k\tau+T]$, $x\in\R$, $y>(2j-1)L$.
		Recalling that \eqref{contact} holds with $\bar y>(2j-1) L$, 
		a contradiction follows from 
		the parabolic strong maximum principle.
		
		We have thereby shown that $u(t,x,y)<\bar u(k\tau + t,y)$ for all $t\geq0$, 
		$(x,y)\in\R^2$.
		Now, the function $\bar u$ satisfies, for $j\in\N$, $j\geq k$,
		$$\bar u(j\tau, 2jL)=\psi(jL,2jL)=
		\frac14+e^{-L}<\frac12$$
		(recall that $L>\ln4$).
		From this and the fact that $u(t,x,y)<\bar u(k\tau + t,y)$ for $t>0$,
		one easily infers that the speed of $u$ satisfies
		$$c(e_2) \leq\lim_{j\to+\infty}\frac{2jL}{j\tau}=\frac{2L}\tau.$$
		
		\smallskip	
		{\em Step 3: there exist $L,M>0$ such that $c(e_1)>c(e_2)$.}
		
		\noindent
		Take $L>8$, so that the conclusions of the Steps 1-2 hold. 
		Hence we can choose $M$ large enough in such a way 
		that $c(e_1)$ is larger than
		the upper bound $2L/\tau$ provided by the Step 2. It follows that 
		$c(e_1)>c(e_2)$. 
	\end{proof}

\begin{proof}[Proof of Proposition~\ref{prop:asymmetric}]
	Let $f_1=f_1(y,u)$ be the function provided by Proposition~\ref{pro:speeds}
	and let $c(e)$ be the speed of the 
	unique (up to shifts in time) pulsating travelling front 
	connecting $1$ to $0$ in the direction $e\in\mathbb{S}^{1}$.
	We know that $c(e_1)>c(e_2)>0$. Fix $c(e_2)<c<c(e_1)$.
	We claim that there exists a bistable reaction term $f_2=f_2(u)$ satisfying
	$f_2'(0)=-1$ and
	such that the homogeneous equation
	\Fi{eq:f2}
		\partial_t u = \Delta u + f_2(u), \quad t \in \R, \ (x,y) \in \R^2 ,
	\Ff
	admits a (unique up to shift) planar front with a speed equal to $c$.
	Such a reaction term can be obtained under the form
	$$f_2(u)=2u\Big(u-\frac12\Big)(1-u)+M'\chi_2(u),$$
	for a suitable choice of $M'$.
	Indeed, for any $M'\geq0$, \eqref{eq:f2} admits a unique planar front, see~\cite{AW},
	and it is not hard to check that its speed $c_{M'}$
	depends continuously on~$M'$.
	To conclude, we observe that $c_0=0$~\cite{AW} and that $c_{M'}\to+\infty$ as  
	$M'\to+\infty$, as we have shown in the Step 1 of the proof
	of Proposition~\ref{pro:speeds}. We point out that the proof of
	Lemma~\ref{lem:uniqueness} still works for the homogeneous equation~\eqref{eq:f2}.
	Namely, the planar front is the unique pulsating travelling front for~\eqref{eq:f2} (up to 
	shift in time or space).
		
	We can now define the reaction $f$ as follows:
	$$f(y,u):=\begin{cases}
				f_1(y,u) & \text{if }0\leq u\leq 1,\\
				f_2(u-1) & \text{if }1<u\leq 2.
			 \end{cases}$$		 
	This function is of class $C^1$ because, we recall, 
	$\partial_u f_1(y,1)=f_0'(1)=-1= f_2 ' (0)$. 
	Moreover, it is a superposition of two reaction terms which are bistable in the sense of 
	Assumption~\ref{ass:bi}, due to Lemmas \ref{lem:f0},
	\ref{lem:f1}.
	Let us show that $f$ satisfies Assumption~\ref{ass:multi} with $I=2$
	and $p_0\equiv\bar p\equiv2$, $p_1\equiv1$, $p_2\equiv0$.
	
	We claim that any periodic steady state $q$ satisfying $0<q<2$ and $q \not \equiv 1$ is linearly unstable. By Lemmas~\ref{lem:f0} and~\ref{lem:f1}, we only need to consider the case when $\min q<1<\max q$. Assume by contradiction that such a~$q$ is not linearly unstable.
	Because the equation is invariant in the direction $e_1$,
	the Step 1 of the proof of Lemma~\ref{lem:f1} implies that $q$ is $x$-independent,
	i.e., $q=q(y)$.
	On the level set $q=1$ we necessarily have that $q'\neq0$, because otherwise 
	$q\equiv1$. Then, the function $q$ being periodic,
	there exists $\eta\in\R$ such that $q(\eta)=1$ and $q'(\eta)>0$.
	Let $a<\eta$ be such that $q'(a)=0$ and $q'>0$ in $(a,\eta)$.
	Then in $(a,\eta)$ there holds that 
	$-q''=f_1(y,q)\geq f_0(q)$.
	Multiplying this inequality by $q'$ and integrating on $(a,\eta)$ we get
	$$-\frac12(q')^2(\eta)\geq\int_a^\eta f_0(q(y))q'(y) dy =\int_{q(a)}^1 f_0(u)du.$$ 
	This is impossible, because $\int_s^1 f_0>0$ for any $s\in[0,1]$, by definition of the function~$f_0$.	
	The claim is proved.

	Summing up, we know that all periodic steady states of~\eqref{eq:parabolic}
	are linearly unstable, excepted for the constant states $0,1,2$ which are linearly stable.
	As shown in the proof of Lemma~\ref{lem:f0}, between any pair of 
	linearly unstable periodic steady states $q<\tilde q$ there must exists a 
	periodic steady state which is not linearly unstable. This implies that
	Assumption~\ref{ass:multi} holds, as announced. It entails
	Assumption~\ref{ass:speeds} too, owing to Proposition~\ref{prop:counter} in the Appendix.
	
	We are in the position to apply Theorem~\ref{th:multi}.
	This provides us with a propagating terrace in any direction~$e\in\mathbb{S}^{1}$.
	Two situations may occur: either the terrace reduces to 
	one single front connecting $2$ to $0$,
	or it consists of two fronts, 
	one connecting $2$ to $1$ and the other connecting 	$1$~to~$0$.
	In the latter case, we have by uniqueness
	that the two fronts are respectively given 
	(up to translation in time)  by the unique planar front for~\eqref{eq:f2} increased by $1$,
	which has speed $c$, and by the unique pulsating front of Proposition~\ref{pro:speeds},
	having speed $c(e)$.
	This case is ruled out if $c>c(e)$ because this violates the condition on the 
	order of the speeds
	of the propagating terrace, see Definition~\ref{def:terrace}. Therefore, when $c > c(e)$ the terrace consists of a single front connecting 2 to 0, and proceeding as in the proof  of Lemma~\ref{lem:uniqueness}, one can show that this front is unique up to time shift.

	Conversely, let us show that if $c\leq c(e)$ then the case of a single front is 
	forbidden. Suppose that there exists 
	a pulsating travelling front $\tilde u$ connecting $2$ to $0$ in the direction~$e$
	with some speed $\tilde c$. Observe that the argument for the 
	uniqueness result in the proof of Lemma~\ref{lem:uniqueness}
	still works if $U_2(\cdot,-\infty)\geq1$ or if $U_1(\cdot,+\infty)\leq0$.
	Hence, on one hand,
	applying this argument with $u_1$ equal to the front connecting $1$ to $0$ and with 
	$u_2=\tilde u$ we get $\tilde c>c(e)$.
	On the other hand, taking $u_1=\tilde u-1$ and $u_2$ equal to the planar front 
	for~\eqref{eq:f2} yields $\tilde c<c$. We eventually infer that $c>c(e)$, a contradiction. Therefore, when $c \leq c(e)$, a terrace necessarily consists of two fronts, and as we pointed out above each of them is unique up to time shift.
	
	We have proved that there exists a unique 
	propagating terrace in any given direction~$e\in\mathbb{S}^{1}$
	and that it consists of two fronts if and only if $c\leq c(e)$.
	This concludes the proof of the proposition because $c(e_2)<c<c(e_1)$.
\end{proof}

\appendix
 
\setcounter{section}{1}
\setcounter{theo}{0}
\section*{Appendix}

Here we recall the order interval trichotomy of Dancer and Hess~\cite{Tri}; see also~\cite{Matano}.
\begin{theo}[\cite{Tri}]\label{DH}
Let $p<p'$ be two periodic steady states of \eqref{eq:parabolic}. Then 
one of the following situations occurs:
\begin{enumerate}[$(a)$]
	\item there is a periodic steady state $\tilde p$ satisfying $p<\tilde p<p'$,
	\item there exists an entire solution $u$ to \eqref{eq:parabolic} such that
	$(u(k,\cdot))_{k\in\Z}$ is an increasing family of periodic functions 
    satisfying
	$$u(-k,\cdot)\searrow p,\qquad 
	u(k,\cdot)\nearrow p',\qquad
	\text{as $k\to+\infty$, uniformly in }[0,1]^N,$$
	\item there exists an entire solution $u$ to \eqref{eq:parabolic} such that
	$(u(k,\cdot))_{k\in\Z}$ is a decreasing family of periodic functions 
	satisfying
	$$u(-k,\cdot)\nearrow p',\qquad 
	u(k,\cdot)\searrow p,\qquad
	\text{as $k\to+\infty$, uniformly in }[0,1]^N.$$
\end{enumerate}
\end{theo}
This trichotomy plays a crucial role in our proofs, as it allows us to look at 
multistable equations as juxtapositions of monostable problems.
Owing to Theorem~\ref{thm:mono}
quoted from Weinberger~\cite{W02},
we infer the existence of the minimal speeds of fronts above and below
any unstable steady state $q$.
In Assumption~\ref{ass:speeds} we require that such speeds are strictly ordered. 
In the next proposition we show that a sufficient condition guaranteering
this hypothesis is that 
$q$ is linearly unstable. We also point out for completeness
that the order between the speeds is always true in the large sense.
\begin{prop}\label{prop:counter}
Assume that $u \mapsto f (x,u)$ is of class $C^1$.

Under either Assumption \ref{ass:bi} or \ref{ass:multi}, and with the notation of 
Assumption~\ref{ass:speeds},
for any unstable periodic steady state $q$ between $0$ and $\bar p$ and any $e\in\mathbb{S}^{N-1}$, there holds~that
$$\overline{c}_q \geq 0 \geq \underline{c}_q.$$
Moreover, if $q$ is linearly unstable, then
$$\overline{c}_q > 0 > \underline{c}_q.$$
\end{prop}

\begin{proof}
We show the inequalities for $\overline{c}_q$, the ones for $\underline{c}_q$
follow by considering the nonlinear term $-f(x,-u)$ and the direction $-e$.

We recall that $\overline{c}_q$ is the minimal speed  of fronts in the direction 
$e$ connecting $p_{i_1}$ to~$q$, where
$p_{i_1}$ is the smallest stable periodic steady state lying above $q$.
Let $\lambda_0$ denote 
the periodic principal eigenvalue of the linearized operator
$$\mathcal{L}_0w := \text{div} (A(x)\nabla  w ) + 
 \partial_u f (x, q(x)) w.$$
The instability of $q$ implies that $\lambda_0\geq0$. We distinguish two cases.

\smallskip
 {\em Linearly unstable case: $\lambda_0>0$}.
\newline
Because the operator $\mathcal{L}_0$ is self-adjoint,
it is well-known that $\lambda_0$ can be approximated by the 
Dirichlet principal eigenvalue
of $\mathcal{L}_0$ in a large ball (see, e.g., \cite[Lemma 3.6]{BHRoques}). Namely, calling 
$\lambda(r)$ the 
principal eigenvalue of $\mathcal{L}_0$ in $B_r$ with Dirichlet boundary condition,
there holds that $\lambda(r)\to\lambda_0$ as $r\to+\infty$. 
Then we can find $r$ large enough so that $\lambda(r)\!>\!0$.
Let $\varphi$ be the associated principal eigenfunction. 
The function~$\psi$ defined~by 
$$\psi(t,x):=q(x)+\varphi(x)e^{\frac12\lambda(r)t},$$  
satisfies for $t\in\R$, $x\in B_R$,
$$\partial_t\psi- \text{div} (A(x) \nabla \psi )= f (x,q)
+\Big(\partial_u f(x,q)-\frac12\lambda(r)\Big)\varphi(x)e^{\frac12\lambda(r)t}.$$
Hence, by the $C^1$ regularity of $u\mapsto f(x,u)$, there exists $T\in\R$ such that 
$\psi$ is a subsolution of~\eqref{eq:parabolic} for $t\leq T$, $x\in B_r$.
Up to reducing $T$, we further have that $\psi < p_{i_1}$ for all $t \leq T$.

Assume by way of contradiction that \eqref{eq:parabolic} admits a pulsating front
$U ( x , x\cdot e -ct)$ connecting $p_{i_1}$ to $q$ with a speed $c\leq0$. 
Let $\xi\in\Z^N$ be such that $U ( \xi , \xi\cdot e-cT)<\psi(T,0)$.
Observe that $U ( x , x\cdot e -ct)$ is bounded from below away from $q$ for
$t\leq T$ and $x\in\overline B_r(\xi)$, because $c\leq0$ and 
$U(\cdot,-\infty)\equiv p_{i_1}>q$.
We can then find $T'<T$ such that 
$$\forall x\in\overline B_r(\xi),\quad
U ( x , x\cdot e -cT')>\psi(T',x-\xi).$$ 
Because $\psi(t,x-\xi)$ is a subsolution of~\eqref{eq:parabolic}
for $t<T$ and $x\in B_r(\xi)$,
which is equal to $q (x)$ for~$x\in\partial B_r(\xi)$,
the comparison principle eventually yields
$$\forall\,T'\leq t\leq T,\ x\in\overline B_r(\xi),\quad
U ( x , x\cdot e -ct)>\psi(t,x-\xi),$$ 
contradicting $U ( \xi , \xi\cdot e-cT)<\psi(T,0)$. 
This shows that $\overline{c}_q > 0$ in this case.

\smallskip

 {\em Case $\lambda_0=0$}.
 \newline
The definition of~$p_{i_1}$, together with either Assumption~\ref{ass:bi} or~\ref{ass:multi},
imply that the case~$(b)$ is the only possible one in
Theorem~\ref{DH} with $p=q$ and $p'=p_{i_1}$. 
Let $u$ be the corresponding entire solution.
For $\sigma\in\R$, let
$\lambda_{\sigma}$ and $\varphi_{\sigma}$ denote 
the periodic principal eigenvalue and eigenfunction of the operator
$$\mathcal{L}_{\sigma} w := \text{div} (A(x)\nabla  w ) + 
2\sigma e A(x)\nabla  w +\left( \sigma^2 e A(x) e +
\sigma \text{div} (A(x)e) +
 \partial_u f (x, q(x)) \right) w.$$
Fix $\eps>0$. We define the following function:
$$\psi(t,x):=u(t,x)-\varphi_{\sigma}(x) e^{\sigma(x\cdot e+\eps t)}.$$
We compute
$$\partial_t \psi- \text{div} (A(x) \nabla \psi ) =  f (x,u)
-[\sigma\eps+\partial_u f(x,q)-\lambda_{\sigma}]\varphi_{\sigma}(x) 
e^{\sigma(x\cdot e+\eps t)}.$$
For $\sigma>0$, there exists $\delta>0$ depending on $\eps,\sigma$
such that $q+\delta<p_{i_1}$ and moreover, for $0\leq s_1\leq s_2\leq\delta$,
there holds that
$$\forall x\in[0,1]^N,\quad 
f(x,q(x)+s_2)-f(x,q(x)+s_1)\leq\Big(\partial_u f(x,q(x))+\frac12\sigma\eps\Big)(s_2-s_1).$$
Then take $k\in\Z$, also depending on $\eps,\sigma$, in such a way that
$$\forall t\leq k,\ x\in[0,1]^N,\quad
u(t,x)\leq q(x)+\delta.$$
We deduce that, for $t<k$ and $x\in\R^N$ such that $\psi(t,x)>q(x)$, the following holds:
$$\partial_t \psi- \text{div} (A(x) \nabla \psi )\leq f (x,\psi)-
\-\left[\frac12\sigma\eps-\lambda_{\sigma}\right]
\varphi_{\sigma}(x) e^{\sigma(x\cdot e+\eps t)}.$$
Now, for $r>0$, call as before $\lambda(r)$ and $\varphi$ the Dirichlet
principal eigenvalue and eigenfunction of $\mathcal{L}_0$ in $B_r$.
Direct computation shows that for $\sigma\in\R$,
$\varphi(x) e^{-\sigma x\cdot e}$ is the Dirichlet principal eigenfunction
of $\mathcal{L}_{\sigma}$ in $B_r$, with eigenvalue $\lambda(r)$. 
It follows that $\lambda(r)<\lambda_\sigma$, because otherwise
$\varphi_{\sigma}$ would contradict the properties of this principal eigenvalue.
Because $\lambda(r)\to\lambda_0=0$ as $r\to+\infty$, 
we deduce that $\lambda_\sigma\geq\lambda_0=0$. 
Namely, $\sigma \mapsto \lambda_\sigma$ attains its minimal value $0$ at
$\sigma=0$ and thus, being regular (see \cite{Kato}) it satisfies
$\lambda_{\sigma}\leq C\sigma^2$ for some $C>0$ and, say, $|\sigma|\leq1$
(this inequality can also be derived 
using the min-max formula of~\cite[Theorem~2.1]{Nadin}).
As a consequence, taking $\sigma=\eps^2$ we find that, for $\eps$ smaller than
some $\eps_0$, the function~$\psi$ is a subsolution of~\eqref{eq:parabolic} for the values~$(t,x)$
such that $t<k$ and $\psi(t,x)>q(x)$.

Assume now by contradiction that there is a pulsating front
$U ( x , x\cdot e -ct)$ connecting~$p_{i_1}$ to $q$ with a speed $c<-\eps$
and $\eps<\eps_0$.
Up to translation in time, it is not restrictive to assume that
$U ( 0 , -ck)<u(k,0)$.
Let $R\in\R$ be such that $U(x,z)>q+\delta$ for $x\in\R^N$ and $z\leq R$.
It follows that $U ( x , x\cdot e -ct)\geq\psi(t,x)$ for $t\leq k$ and $x\cdot e-ct\leq R$.
On the other hand, we see that
$$\forall t<k,\ x\cdot e-ct\geq R,\quad
\psi(t,x)\leq\delta-\left(\min\varphi_{\sigma}\right)e^{\eps^2(R+ct+\eps t)}.
$$
The right-hand side goes to $-\infty$ as $t\to-\infty$ because $c+\eps<0$.
We can then find $T<k$ such that $U ( x , x\cdot e -ct)\geq\psi(t,x)$ for 
all $t\leq T$ and $x\in\R^N$. Hence, because $U>q$, we can apply
the comparison principle and infer that 
$U ( 0 , -ck)\geq u(k,0)$, which is a contradiction.
We have shown that fronts cannot have a speed smaller than $-\eps$,
for $\eps$ sufficiently small, whence $\overline{c}_q \geq 0$.
\end{proof}

%-------------------------------------------------------------------------------

%-------------------------------------------------------------------------------
%-------------------------------------------------------------------------------
%

\end{document}